\documentclass[a4paper, 11pt]{article}
\usepackage{amsmath, amssymb, mathrsfs}
\usepackage[all]{xy} 

\usepackage{graphicx, xcolor}
\usepackage{tikz}
\usetikzlibrary{intersections, calc, arrows.meta}

\usepackage{newtxtext, newtxmath}

\usepackage{bm}

\usepackage{extarrows}

\makeatletter
\@addtoreset{equation}{section}
\makeatother

\usepackage{theorem}
{\theorembodyfont{\upshape}\newtheorem{proof}{Proof.}}
\newtheorem{theo}{Theorem}[section]
{\theorembodyfont{\upshape}\newtheorem{defi}[theo]{Definition}}
\newtheorem{lemm}[theo]{Lemma}
\newtheorem{prop}[theo]{Proposition}
\newtheorem{cor}[theo]{Corollary}

{\theorembodyfont{\upshape}\newtheorem{rem}[theo]{Remark}}
{\theorembodyfont{\upshape}\newtheorem{exam}[theo]{Example}}

{\theorembodyfont{\upshape}\newtheorem{shomei}{Proof of Theorem\,\ref{thm1.1}}}

{\theorembodyfont{\upshape}\newtheorem{shomeip}{Proof of Corollary\,\ref{cor1.2}}}
 
{\theorembodyfont{\upshape}\newtheorem{shomeio}{Proof of Proposition\,\ref{prop4.6}}}

{\theorembodyfont{\upshape}\newtheorem{shomeiu}{Proof of Proposition\,\ref{prop4.16}}}

{\theorembodyfont{\upshape}\newtheorem{shomeiy}{Proof of Proposition\,\ref{prop4.10}}}

\def\qed{\hfill $\Box$}

\DeclareMathOperator{\pr}{pr}
\DeclareMathOperator{\Lie}{Lie}
\DeclareMathOperator{\M}{\mathcal{M}}
\DeclareMathOperator{\D}{\mathcal{D}}

\DeclareMathOperator{\W}{\mathsf{W}}
\DeclareMathOperator{\T}{\mathsf{RibTree}} \DeclareMathOperator{\Tt}{\mathsf{RibTree}^{\mathsf{3}}}
\DeclareMathOperator{\Tc}{\mathsf{Corolla}}
\DeclareMathOperator{\p}{\mathsf{cont}}
\DeclareMathOperator{\fr}{\mathsf{f}_{\mathsf{re}}} \DeclareMathOperator{\mg}{\mathsf{m}_{\mathsf{G}}} \DeclareMathOperator{\ms}{\mathsf{m}_{\mathsf{SO(3)}}}

\DeclareMathOperator{\OO}{\mathsf{O}} \DeclareMathOperator{\PP}{\mathsf{P}} \DeclareMathOperator{\f}{\mathsf{f}} \DeclareMathOperator{\g}{\mathsf{g}}

\DeclareMathOperator{\Li}{\mathcal{L}} \DeclareMathOperator{\C}{\mathcal{C}}
\DeclareMathOperator{\Oh}{\mathcal{O}}
\DeclareMathOperator{\R}{\mathbb{R}}
\DeclareMathOperator{\Z}{\mathbb{Z}}

\newcommand{\He}{H_{\mathrm{ext}}}
\newcommand{\Ei}{E_{\mathrm{int}}}

\newcommand{\fk}{\mathsf{f}_{\mathsf{K}\ddot{\mathsf{a}}\mathsf{h}}}
\newcommand{\unit}{\mbox{1}\hspace{-0.25em}\mbox{l}}

\usepackage{titlefoot}

\title{Operad Structures in Geometric Quantization \\of the Moduli Space of Spatial Polygons}
\author{Yuya Takahashi}
\date{}

\begin{document}

\maketitle

\begin{abstract}
The moduli space of spatial polygons is known as a symplectic manifold equipped with both K\"ahler and real polarizations. In this paper, associated to the K\"ahler and real polarizations, 
morphisms of operads $\fk$ and $\fr$ are constructed by using the quantum Hilbert spaces $\mathscr{H}_{\mathrm{K}\ddot{\mathrm{a}}\mathrm{h}}$ and $\mathscr{H}_\mathrm{re}$, respectively. Moreover, the relationship between the two morphisms of operads $\fk$ and $\fr$ is studied and then the equality $\dim \mathscr{H} _{\mathrm{K}\ddot{\mathrm{a}}\mathrm{h}}=\dim \mathscr{H}_\mathrm{re}$ is proved in general setting. This operadic framework is regarded as a development of the recurrence relation method by Kamiyama\,\cite{kami} for proving $\dim \mathscr{H}_{\mathrm{K}\ddot{\mathrm{a}}\mathrm{h}}=\dim \mathscr{H}_\mathrm{re}$ in a special case. 

\unmarkedfntext{2020 Mathematics Subject Classification. Primary 53D50; Secondary 53D20, 18M60.}
\unmarkedfntext{Key Words and Phrases. Geometric quantization, Momentum maps; symplectic reduction, Operads (general).}
\end{abstract}

\section{Introduction}

Geometric quantization is one formulation of quantization in terms of symplectic geometry. The first step in the formulation is to construct a certain vector space $\mathscr{H}$ called a quantum Hilbert space from a given symplectic manifold $(M,\,\omega)$. From a general framework of the Souriau-Kostant prequantization (see \cite[Subsection 2.2]{kiri} for example), the symplectic structure $\omega$ is required to define an integral cohomology class, which enables us to take a prequantum line bundle $L\to (M,\,\omega)$, that is, a complex line bundle $L$ over $M$ such that the first Chern class is given by the cohomology class of $\omega$. Moreover, an additional datum called a {\it polarization} is needed to construct the quantum Hilbert space $\mathscr{H}$ and then the vector space is defined as a space of flat sections of $L$ along the polarization. Based on the perspective of physics, whatever polarization we choose, the resulting vector space $\mathscr{H}$ is believed to be unique up to isomorphism. We refer to this guiding principle as the principle of ``invariance of polarization'' after Guillemin and Sternberg\,\cite{GC}. 

Among polarizations, the following two types are important: a {\it K\"ahler polarization} and a {\it real polarization}. A {\it K\"ahler polarization} is given by a compatible complex structure of $(M,\,\omega)$ and then, the quantum Hilbert space is defined to be the space of holomorphic sections of $L$
\[
  \mathscr{H}_{\mathrm{K}\ddot{\mathrm{a}}\mathrm{h}}\,=\, H^0(M,\,\mathcal{O}_L).
\]

On the other hand, a {\it real polarization} is given by a (singular) Lagrangian fibration $\pi:M\to B$ over a manifold $B$ of the half real dimension of $M$ and then, the quantum Hilbert space is defined to be 
\[
  \mathscr{H}_\mathrm{re}\,=\, \bigoplus _{p\in \mathrm{Im}(\pi)} \,\Gamma_{\mathrm{flat}}(L|_{{\pi}^{-1}(p)}),
\]
where $\Gamma_{\mathrm{flat}}(L|_{\pi^{-1}(p)})$ is the space of global flat sections of the restriction of $L$ to $\pi^{-1}(p)$. A point $p\in \mathrm{Im}(\pi)$ is called a Bohr-Sommerfeld
point if the space $\Gamma_{\mathrm{flat}}(L|_{\pi^{-1}(p)})$ is non-trivial. Let $BS$ denote the set of Bohr-Sommerfeld points. Now, by the principle of ``invariance of polarization'', we expect the following equalities
\[
	\dim H^0(M,\,\mathcal{O}_L)\,=\,\dim\mathscr{H}_{\mathrm{K}\ddot{\mathrm{a}}\mathrm{h}}\,=\,\dim\mathscr{H}_\mathrm{re}\,=\,\#\,BS
\]
and in fact, this is observed rigorously in several cases. The typical example is the case when $M$ is a toric manifold, where we consider a toric manifold as a symplectic manifold equipped with both K\"ahler and real polarizations by its canonical K\"ahler structure and the momentum map of the Hamiltonian torus action. The cases when $M$ is an ``almost'' toric manifold such as a complex flag manifold with the Gelfand-Cetlin system\,\cite{GC} or the moduli space of $SU(2)$-flat bundles on a compact Riemann surface with the Goldman system\,\cite{JW} are also known. In these examples, a Bohr-Sommerfeld point can be characterized as a lattice point in (the closure of) the moment polytope. In this paper, we focus on the case of ``almost'' toric manifolds called the {\it moduli space of spatial polygons with the bending system}. From now on, we consider $\dim\mathscr{H}_\mathrm{re}$ as the number of lattice points in (the closure of) the moment polytope. 

Let $n\geq 3$ and $\bm{r}=(r_0,\,\dots,\,r_{n-1})\in\mathbb{R}^n _{> 0}$. The {\it moduli space of spatial $n$-gons with edge-lengths $\bm{r}$} or simply the {\it polygon space} is defined as the following space
\[
  \M(\bm{r}) =\bigl\{\bm{u}=(u_0,\,\dots,\,u_{n-1})\in S^2(r_0)\times\cdots \times S^2(r_{n-1})\,\bigl|\, u_0 +\cdots + u_{n-1} = 0\bigr\}\bigl/{SO(3)},
\]
where $S^2(r_i)$ is a sphere of radius $r_i$ in $\mathbb{R}^3$ with the standard $SO(3)$-action and the quotient is taken by the diagonal action. (Unless otherwise noted, we assume that $\M(\bm{r})$ is not empty.) Here we assume 
\begin{equation}
  \pm r_0\pm\cdots\pm r_{n-1}\neq 0, \label{eq0000}
\end{equation}
which guarantees that $\M(\bm{r})$ is a smooth manifold of real dimension $2n-6$.
Then the integral condition on the edge-length $\bm{r}\in\mathbb{Z}^n _{> 0}$ together with the condition\,(\ref{eq0000}) endows the polygon space $\M(\bm{r})$ with a natural setting of geometric quantization via a K\"ahler polarization, namely a K\"ahler structure and a prequantum line bundle $\Li(\bm{r})\to \M(\bm{r})$ (see Subsection 3.1).

On the other hand, a real polarization on the polygon space was introduced by Kapovich and Millson\,$\cite{kapo}$. They considered the functions 
\begin{equation}
b_i : \M(\bm{r})\longrightarrow\R\,\,;\,\,\,[\bm{u}]\longmapsto\|u_0+\cdots+u_i\|\label{eq01}
\end{equation}
of the $i$-th diagonal length for $i=1,\dots,n-3$ and constructed a Hamiltonian $(n-3)$-torus action on an open dense subset $\M^{\prime}(\bm{r})$ of $\M(\bm{r})$ such that the momentum map is given by the restriction of the following map to $\M^{\prime}(\bm{r})$
\begin{equation}
\pi^{\bm{r}} =(b_1,\dots,b_{n-3}):\M(\bm{r})\longrightarrow \R^{n-3},\label{eq00}
\end{equation}
which is called the {\it bending system}. In this sense, polygon space $\M(\bm{r})$ with the bending system can be considered as an ``almost'' toric manifold and then the number of lattice points in the closure of the moment polytope is given by $\#\ \mathrm{Im} (\pi^{\bm{r}}) \cap \Z^{n-3}$. 

The equation $\dim\mathscr{H}_{\mathrm{K}\ddot{\mathrm{a}}\mathrm{h}}=\dim\mathscr{H}_\mathrm{re}$ on the polygon space $\M(\bm{r})$ was first obtained by Kamiyama\,\cite{kami}, when $n\geq 5$ is odd and $\bm{r}=(1,\dots,1)$. We note that the condition\,(\ref{eq0000}) on the edge-lengths is automatically satisfied in this case. 

\begin{theo}{\bf (Kamiyama\,\cite[Theorem\,A]{kami})}\label{thm0}
Suppose that $n\geq 5$ is odd and let $\M_n=\M(1,\dots,1)$, $\Li_n=\Li (1,\dots,1)$, and $\pi_n=\pi^{(1,\dots,1)}$. Then we have
	 \[\dim H^0\bigl(\M_n,\Oh_{\Li_n}\bigr) = \#\ \mathrm{Im} (\pi_n) \cap \Z^{n-3}.\]
\end{theo}

He first derived the recurrence relation\,(\ref{eq4.20}) for the right hand side by changing two integer-valued parameters of the polygon space: the number $n$ of edges and the first edge-length $i$. Then he showed that the left hand side also satisfies the same recurrence relation and hence obtained Theorem\,\ref{thm0}. 

The bending system\,(\ref{eq00}) is known to be generalized to a map $\pi^{\bm{r}} _T:\M(\bm{r})\to\R^{n-3}$ associated to any triangulation $T$ of $n$-gons (see \cite{kapo}, Subsection\,4.1 for details). The aim of this paper is to generalize Theorem\,\ref{thm0} for more general $\bm{r}\in\Z_{>0} ^n$ and any triangulation $T$ of $n$-gons. However, Kamiyama's argument above can not be applied to the case of any triangulation $T$ of $n$-gons literally. The reason is as follows. The bending system\,(\ref{eq00}) coincides with the map $\pi^{\bm{r}} _T:\M(\bm{r})\to\R^{n-3}$ when $T$ is a special triangulation of $n$-gons given by the $(n-1)$-caterpillar in Example\,\ref{ex2.6}. 
Then the set $\{\text{the }n\text{-caterpillar}\}_{n\geq 4}$ has a canonical linear order by the number $n\geq 4$, which can be regarded as a key point for deriving Kamiyama's recurrence relation\,(\ref{eq4.20}) on $n\geq 4$ and $i\geq 0$. In contrast, the set of all triangulations does not have such a linear order. 

Our idea to generalize Theorem\,\ref{thm0} is to use the notion of operads after identifying a triangulation of $n$-gons with its dual graph: a trivalent rooted ribbon $(n-1)$-tree (see Section\,2 for relevant terminology). 
In this paper, we will consider the operad of trivalent rooted ribbon trees, where its grafting operation plays the role of the linear order in the caterpillar case. Then we can describe  ``recursive structures'' arising from the number of lattice points by morphisms of operads, which are main key players in this paper. Indeed, from our framework of operads Kamiyama's recurrence relation\,(\ref{eq4.20}) is replaced by the relations\,(\ref{eq4.21}) and (\ref{eq4.22}). 

Now we will state the main result. We denote by $\Tt=\{\Tt(n)\}_{n\geq 1}$ the trivalent rooted ribbon tree operad (Example\,$\ref{ex2.8}$) and by $\Tc=\{\Tc(n)\}_{n\geq 1}$ the corolla operad (Example\,$\ref{ex2.9}$). Let $\W(\Z_{\geq 0})$ be a certain operad given in Definition\,$\ref{def2.12}$, which consists of integer-valued functions on a product space of $\Z_{\geq 0}$. Then, the main result in this paper is the following.

\begin{theo}\label{thm1.1}
We can associate non-trivial morphisms of operads $\fk:\Tc\to\W(\Z_{\geq 0})$ and $\fr:\Tt\to\W(\Z_{\geq 0})$ to the K\"ahler and real polarizations on the polygon spaces respectively. 
	Furthermore, the morphism $\fr$ coincides with the pull-back of the morphism $\fk$ by a natural morphism $\p:\Tt\to\Tc$ given in Example\,\ref{ex2.14}. In other words, we have the following commutative diagram:
	\[
  \xymatrix@C=36pt{
                   \Tt \ar[r]^{\fr} \ar[d]_{\p}  &\W(\Z_{\geq 0}) \\
                   \Tc \ar[ru]_{\fk}   &\\
                   }
\]
\end{theo}

Theorem\,\ref{thm1.1} yields the following corollary.

\begin{cor}\label{cor1.2}
	Let $n\geq 4$, let $T$ be any trivalent rooted ribbon $(n-1)$-tree (or any triangulation of $n$-gons), and let $\bm{r}$ be any $n$-tuple of positive integers satisfying the condition\,$(\ref{eq0000})$. Then we have
	\begin{equation}
	  \dim H^0\bigl(\M(\bm{r}),\Oh_{\Li(\bm{r})}\bigr)\,=\, \#\ \mathrm{Im} \bigl(\pi^{\bm{r}} _T\bigr) \cap \Z^{n-3}.\label{eq000}
	\end{equation}
\end{cor}

As we already mentioned, the condition\,$(\ref{eq0000})$ on the edge-lengths is necessary for the polygon space $\M(\bm{r})$ to be a smooth manifold. In this sense, Corollary\,\ref{cor1.2} completely generalizes Theorem\,\ref{thm0}. 

This paper is organized as follows. After recalling and introducing basic preliminaries including the definitions of the three operds $\Tc$, $\Tt$, and $\W(\Z_{\geq 0})$ in Section\,2, 
we construct the morphisms $\fk$ and $\fr$ in Sections\,3 and 4 respectively. In Section\,5, we complete the proofs of Theorem\,\ref{thm1.1} and Corollary\,\ref{cor1.2}. 

{\bf Acknowledgments.} I am very grateful to my supervisor, Prof.\ Hiroshi Ohta for his guidance and helpful discussions. In particular, his comments led me to consider any trivalent tree in generalizing Theorem\,\ref{thm0}, which is indispensable for our operadic formulation. I would also like to thank him for suggesting many corrections in the manuscript. I also wish to thank Prof.\ Yasuhiko Kamiyama, Prof.\ Hiroshi Konno, Prof.\ Yuichi Nohara, Prof.\ Tatsuru Takakura, and Prof.\ Takahiko Yoshida for helpful comments in my talk about this paper in research seminars.

\section{Preliminaries}
In Section\,2, we recall and fix terminology on trees and operads, and give examples for the later sections. We refer to \cite{Markl2,Markl} for the materials in this section. 
\subsection{Rooted ribbon trees}

A {\it tree} is a contractible CW-complex of dimension 1. All trees in this paper are assumed to be compact. A {\it vertex} (resp.\,\,an {\it edge}) of a tree is a $0$-cell (resp. an open $1$-cell). A {\it half-edge} is a connected subspace consisting of one edge and one vertex. A vertex $v$ and a half-edge $h$ (resp. an edge $e$) are called adjacent if $h$ contains $v$ (resp. $e\cup v$ is connected). For any half-edge $h$, we denote by $-h$ the unique half-edge such that $h\cup (-h)$ is the closure of some edge. For a tree $T$, we denote by $V(T)$ the set of vertices, by $E(T)$ the set of edges, and by $H(T)$ the set of half-edges. In addition, for any vertex $v$, we denote by $H_v(T)$ the set of half-edges adjacent to $v$ and by $val(v)$ the number $\#H_v(T)$, which is called the {\it valence of $v$}. From now on, we assume that all trees have no vertex of valence $2$. First, we define a {\it ribbon tree}.

\begin{defi}	
	A {\it ribbon tree} is a tree $T$ where every vertex $v\in V(T)$ is equipped with a cyclic order on $H_v(T)$, namely a cyclic permutation $\sigma_v\in\mathfrak{S}_{H_v(T)}$ of length $val(v)$.   
\end{defi}

A vertex $v$ is called {\it external} if $val(v)=1$ and {\it internal} otherwise. A half-edge is called {\it external} if the adjacent vertex is external, and an edge is {\it internal} if all the adjacent vertices are internal. We denote by $\He(T)$ the set of external half-edges and by $\Ei(T)$ the set of internal edges. 
Next, we introduce a {\it rooted ribbon tree}.
\begin{defi} 
	(1) A {\it rooted ribbon tree} is a ribbon tree $T$ with a distinguished external half-edge, called the {\it root} and denoted by $r_T$. The adjacent vertex to the root is called the {\it root vertex}, denote by $v_T$. The external half-edges except the root are called the {\it leaves}. We call a rooted ribbon tree with $n$ leaves simply a {\it rooted ribbon $n$-tree}. 

	\noindent
	(2) Let $T$ and $S$ be rooted ribbon trees. Then $T$ is {\it isomorphic to $S$} if there exists an isomorphism $f : T\to S$ of CW-complexes which preserves the roots and the cyclic orders at each vertex. 
\end{defi}

A rooted ribbon tree has a natural numbering on leaves due to its ribbon structure.

\begin{defi}\label{defi2.4}
(1) Let $T$ be a ribbon tree. We define the permutation $\iota$ on $H(T)$ as 
	  \[\iota(h)\,=\,\sigma_v (-h), \]
	where $v$ is the vertex adjacent to the half-edge $-h$. 
	The set $\He(T)$ of external half-edges has a canonical cyclic order $\tau$ given by  
	 \[\tau(h)\,=\,\iota^{N_h}(h),\] 
	where $N_h$ is the minimum number $N\in\Z_{>0}$ such that the half-edge $\iota^N(h)$ becomes external. 
	
	\noindent
	(2) Let $T$ be a rooted ribbon $n$-tree. For $0\leq i\leq n$, we set 
	\[
	  (r_T)_i\,=\,\tau^i (r_T)
	\]
	and call it the {\it $i$-th external half-edge}, or the {\it $i$-th leaf} \,if $i\neq 0$ (the $0$-th external half-edge is nothing but the root). In addition, we denote by $(v_T)_i$ the adjacent vertex to $(r_T)_i$ and call it the {\it $i$-th external vertex}, or the {\it $i$-th leaf vertex} \,if $i\neq 0$. 
\end{defi}

In this paper, we draw a rooted ribbon tree in the manner that a cyclic order at each vertex becomes compatible to the counterclockwise orientation. Since the numbering on external half-edges is also counterclockwise, a rooted ribbon tree can be described as in e.g. Figure\,\ref{fig1}. 
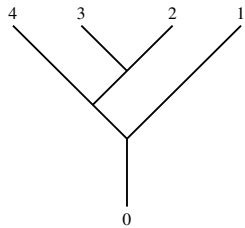
\begin{figure}[!h]
\centering
\begin{tikzpicture}[scale=0.6, transform shape, line width=0.7pt]
 \draw (5,1.5)--(5,3)--(7.5,5.5);
 \draw (2.5,5.5)--(5,3);
 \draw (6,5.5)--($(2.5,5.5)!(6,5.5)!(5,3)$);
 \draw (4,5.5)--($(6,5.5)!(4,5.5)!(4.5,4)$);
 \draw (5,1.5)node[below]{$0$};
 \draw (2.5,5.5)node[above]{$4$};
 \draw (4,5.5)node[above]{$3$};
 \draw (6,5.5)node[above]{$2$};
 \draw (7.5,5.5)node[above]{$1$};
\end{tikzpicture}
\caption{a rooted ribbon $4$-tree.}\label{fig1}
\end{figure}

There is an operation which produces a new tree from two rooted ribbon trees, called {\it grafting}. 

\begin{figure}[!h]
\centering
\begin{tikzpicture}[scale=0.6, transform shape, line width=0.7pt]
 \begin{scope}[xshift=-0.3cm, yshift=-0.5cm]
 \draw (5,2)--(5,3)--(7,5);
 \draw (3,5)--(5,3);
 \draw (5,3)--(6,5);
 \draw (5,3)--(4,5);
 \draw (5,3)--(5,5);
 \draw (5,2)node[below]{$0$};
 \draw (3,5)node[above]{$5$};
 \draw (4,5)node[above]{$4$};
  \draw (5,5)node[above]{$3$};
 \draw (6,5)node[above]{$2$};
 \draw (7,5)node[above]{$1$};
 \end{scope}
 
 \draw(8,3) circle (0.12)node[below right]{\ 3};
 
 \begin{scope}[xshift=-0.5cm]
 \draw (11,2)--(11,3)--(12,4);
 \draw (10,4)--(11,3);
 \draw (11,3)--(11,4);
 \draw (11,2)node[below]{$0$};
 \draw (10,4)node[above]{$3$};
 \draw (11,4)node[above]{$2$};
 \draw (12,4)node[above]{$1$};
 \end{scope}
 
 \draw (13.5,3)node{$\xlongequal{\ \ \ }$};
 
 \begin{scope}[xshift=0.5cm, yshift=-1.6cm]
 \draw (17,2)--(17,3)--(19,5);
 \draw (15,5)--(17,3);
 \draw (17,3)--(18,5);
 \draw (17,3)--(16,5);
 \draw (17,3)--(17,5)--(17,6)--(17,7);
 \draw (17,6)--(16,7);
 \draw (17,6)--(18,7);
 \draw (17,2)node[below]{$0$};
 \draw (15,5)node[above]{$7$};
 \draw (16,5)node[above]{$6$};
  \draw (18,7)node[above]{$3$};
  \draw (17,7)node[above]{$4$};
 \draw (16,7)node[above]{$5$};
 \draw (18,5)node[above]{$2$};
 \draw (19,5)node[above]{$1$};
 \end{scope}
\end{tikzpicture}
\caption{grafting.}\label{fig2}
\end{figure}
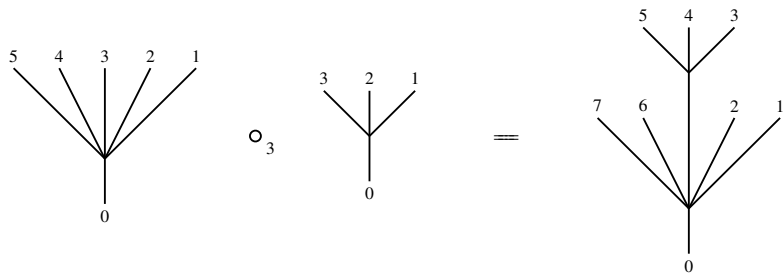

\begin{defi}\label{def1.5}
	Let $T$ and $S$ be rooted ribbon trees, let $n$ (resp. $m$) be the number of leaves of $T$ (resp. $S$), and let $1\leq i\leq n$. We consider the gluing space 
	\[T\circ_i S\,=\,T\sqcup S\bigl/(v_T)_i\sim v_S, \]
	which inherits ribbon tree structure from $T$ and $S$ in the obvious way. 
	We will regard $T\circ_i S$ as a rooted ribbon tree by adopting $r_T\in\He(T\circ_i S)$ as the root and call it the {\it grafted tree of $T$ and $S$ along the $i$-th leaf} (see Figure\,\ref{fig2}). We often identify $T$ and $S$ with the subcomplexes  
	\[
	  T\sqcup(-r_S)\text{ \ and \ } S\sqcup(-(r_T)_i)
	\]
	of $T\circ_i S$ respectively.
	\end{defi}

Here are examples of a rooted ribbon tree.
\begin{exam}\label{ex1.2}
	A rooted ribbon $n$-tree which has no internal edge is called the {\it $n$-corolla} (see Figure\,\ref{fig3}). In particular, the $1$-corolla is called the {\it exceptional tree}, which is the rooted ribbon tree having just one edge. 
	Note that the $n$-corolla is unique up to isomorphisms. 
\end{exam}

\begin{figure}[!h]
\begin{minipage}[b]{0.49\columnwidth}
\centering
\begin{tikzpicture}[scale=0.6, transform shape, line width=0.7pt]
\draw (5,1.5)--(5,3)--(7.5,5.5);
 \draw (2.5,5.5)--(5,3);
 \draw (5,3)--(6,5.5);
 \draw (5,3)--(4,5.5);
 \draw[loosely dotted] (4.4,5.5)--(5.6,5.5);
 \draw (5,1.5)node[below]{$0$};
 \draw (2.5,5.5)node[above]{$n$};
 \draw (4,5.5)node[above]{$n-1$};
 \draw (6,5.5)node[above]{$2$};
 \draw (7.5,5.5)node[above]{$1$};
\end{tikzpicture}
\caption{the $n$-corolla.}\label{fig3}
\end{minipage}
\begin{minipage}[b]{0.49\columnwidth}
\centering
\begin{tikzpicture}[scale=0.6, transform shape, line width=0.7pt]
 \draw (5,1.5)--(5,3)--(7.5,5.5);
 \draw (2.5,5.5)--(5,3);
 \draw (6.2,5.5)--($(2.5,5.5)!(6.2,5.5)!(5,3)$);
 \draw (3.8,5.5)--($(2.5,5.5)!(3.8,5.5)!(5,3)$);
 \draw[loosely dotted] (4.4,5.5)--(5.7,5.5);
 \draw (5,1.5)node[below]{$0$};
 \draw (2.5,5.5)node[above]{$n$};
 \draw (3.8,5.5)node[above]{$n-1$};
 \draw (6.2,5.5)node[above]{$2$};
 \draw (7.5,5.5)node[above]{$1$};
\end{tikzpicture}
\caption{the $n$-caterpillar.}\label{fig4}
\end{minipage}
\end{figure}

\begin{exam}\label{ex2.6}
	Suppose $n\geq 2$. 
	We refer the rooted ribbon $n$-tree in Figure\,\ref{fig4} as the {\it $n$-caterpillar}. Precisely, 
	we define the {\it $n$-caterpillar} recursively: 
	
	\noindent
	(i) when $n=2$, the {\it $2$-caterpillar} is the $2$-corolla,  
	
	\noindent
	(ii) when $n\geq 3$, the {\it $n$-caterpillar} is the grafted tree of the {\it $2$-caterpillar} and the {\it $(n-1)$-caterpillar} along the second leaf.
	
	\noindent
	Note that the $n$-caterpillar is unique up to isomorphisms. We denote by $C_n$ the isomorphism class of $n$-caterpillars.
\end{exam}

When it has an internal edge, a rooted ribbon tree can be decomposed into smaller trees as in Figure\,\ref{fig5}. At the end of this subsection, we state this fact as the next proposition in terms of grafting. 

\begin{figure}[!h]
\begin{minipage}[b]{0.49\columnwidth}
\centering
\begin{tikzpicture}[scale=0.6, transform shape, line width=0.7pt]
 \draw (5,1.5)--(5,3)--(7.5,5.5);
 \draw (2.5,5.5)--(5,3);
 \draw (6,5.5)--($(2.5,5.5)!(6,5.5)!(5,3)$);
 \draw (4,5.5)--($(6,5.5)!(4,5.5)!(4.5,4)$);
 \draw (5,1.5)node[below]{$0$};
 \draw (2.5,5.5)node[above]{$4$};
 \draw (4,5.5)node[above]{$3$};
 \draw (6,5.5)node[above]{$2$};
 \draw (7.5,5.5)node[above]{$1$};
\end{tikzpicture}
\end{minipage}
\begin{minipage}[b]{0.49\columnwidth}
\centering
\begin{tikzpicture}[scale=0.6, transform shape, line width=0.7pt]
\draw (5,1.5)--(5,3)--(7.5,5.5);
 \draw (2.5,5.5)--(5,3);
 \draw (6.5,6)--(5.2,4.7);
 \draw (4.25,3.75)--(4.6,4.1);
 \draw (4.5,6)--($(6.8,6.3)!(4.5,6)!(5.4,4.9)$);
 \draw (5,1.5)node[below]{$0$};
 \draw (2.5,5.5)node[above]{$4$};
 \draw (4.5,6)node[above]{$3$};
 \draw (6.5,6)node[above]{$2$};
 \draw (7.5,5.5)node[above]{$1$};
\end{tikzpicture}
\end{minipage}
\caption{decomposing a rooted ribbon tree.}\label{fig5}
\end{figure}
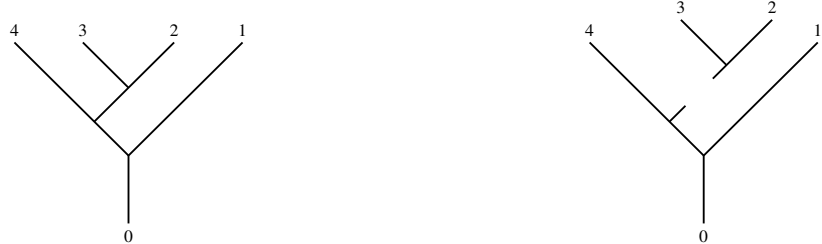

\begin{prop}\label{prop2.6}
	Let $T$ be a rooted ribbon $n$-tree and suppose that $T$ has an internal edge $e$. Then, there exist rooted ribbon trees $S_e$ and $S_e ^{\prime}$ with more than two leaves such that $T$ is isomorphic to the grafted tree $S_e$ and $S_e ^{\prime}$ along some leaf of $S_e$.
\end{prop}
\begin{proof}
	Fix a point $a\in e$. Then we denote by $S_e$ (resp. $S_e ^{\prime}$) the closure of the connected component of $T\setminus\{a\}$ which contains the root $r_T$ (resp. the closure of the other connected component). The subspaces $S_e$ and $S_e ^{\prime}$ become the desired rooted ribbon trees in the obvious way. 
	\qed
\end{proof}

\subsection{Operads}
Here is the definition of an operad we will use in this paper.
\begin{defi}\label{def2.6}
{\it A (non-symmetric) operad (in the category of sets)} is a sequence $\OO=\{\OO(n)\}_{n\geq 1}$ of sets, together with maps called the {\it operadic compositions} 
\[
  \circ_i : \OO(n)\times\OO(m)\longrightarrow\OO(n+m-1)
\]
for $1\leq i\leq n$ and $m\geq 1$. These data fulfill the following axioms.

\noindent
{\it Associativity.} \ For each $1\leq j\leq n$, $m, l\geq 1$, $X\in\OO(n)$, $Y\in\OO(m)$ and $Z\in\OO(l)$, 
\[
  (X\circ_j Y)\circ_i Z\,=\, \left\{
\begin{array}{ll}
(X\circ_i Z) \circ_{j+l-1} Y & \text{if } 1\leq i<j,\\
X\circ_j(Y \circ_{i-j+1} Z) & \text{if } j\leq i<m+j,\\
(X \circ_{i-m+1} Z)\circ_j Y & \text{if } j+m\leq i\leq n+m-1.
\end{array}
\right. 
\]

\noindent
{\it Unitality.} \ There exists an element $\unit \in \OO(1)$ called the {\it unit} such that
\[
  X\circ_i \unit =X \text{\ \ and\ \ } \unit\circ_1 Y = Y
\]
for each $1\leq i\leq n$, $m\geq 1$, $X\in\OO(n)$ and $Y\in\OO(m)$. 
\end{defi}

Here are the examples including operds $\Tc$, $\Tt$, and $\W(\Z_{\geq 0})$ in Theorem\,\ref{thm1.1}.

\begin{exam}\label{ex2.9}
	Here, we identify the isomorphism class of the $n$-corollas with the set $\{n\}$. 
	The {\it corolla operad} is the sequence $\Tc=\{\Tc(n)\}_{n\geq 1}$ given by 
	\[\Tc(n)\,=\, \{n\}\] with the obvious maps ${\circ}_i : \Tc(n)\times \Tc(m)\to \Tc(n+m-1)$. (The composition ${\circ}_i$ means contracting one internal edge after grafting two corollas along the $i$-th leaf.) 
\end{exam}

\begin{exam}\label{ex2.3}
	The {\it rooted ribbon tree operad} is the sequence $\T=\{\T(n)\}_{n\geq 1}$ given by 
	\[\T(n)\,=\, \{\text{\,isomorphism classes of rooted ribbon }n\text{-trees}\,\}\] with the maps $\circ_i : \T(n)\times \T(m)\to \T(n+m-1)$ defined by 
	\[[T]\circ_i[S]\,=\,[T\circ_i S],
\]where $T\circ_i S$ is the grafted tree in Definition\,$\ref{def1.5}$. Note that the unit $\unit$ is given by the exceptional tree in Example\,$\ref{ex1.2}$. 
\end{exam}

\begin{exam}\label{ex2.8}
A {\it trivalent tree} is a tree where the valence of any vertex is equal to $1$ or $3$. The {\it trivalent rooted ribbon tree operad} is the sequence $\Tt=\{\Tt(n)\}_{n\geq 1}$ given by 
	\[\Tt(n)\,=\, \bigl\{\,[T]\in\T(n)\,\bigl|\,T\text{ is trivalent.}\,\bigr\}\] with the maps $\circ_i : \Tt(n)\times \Tt(m)\to \Tt(n+m-1)$ defined as in Example $\ref{ex2.3}$. 
\end{exam}

\begin{defi}\label{def1.10}
	Let $\C$ be a set. Then we define a sequence $\W(\C)=\{\W(\C)(n)\}_{n\geq 1}$ of sets by
	\[\W(\C)(n)\,=\,\bigl\{\,f : \C\times \C^n\to \Z\,\bigl|\,N(f,\bm{c})<\infty \text{ for any }\bm{c}\in\C^n\,\bigr\} ,\]
	where $N(f,\bm{c})$ is the number of elements $d\in\C$ satisfying $f(d;\bm{c})\neq 0$. 
	In addition, we define the maps $\circ_i : \W(\C)(n)\times \W(\C)(m)\longrightarrow \W(\C)(n+m-1)$ by 
	\begin{equation}
	  (f\circ_i g) (d;\bm{c})\,=\,\sum_{k\in\C}\, f(d;\bm{c}^{i,m;k})\cdot g(k;\bm{c}_{i,m})\label{eq1.10}
	\end{equation}
	for any $d\in\C$ and $\bm{c}=(c_1,\dots, c_{n+m-1})\in\C^{n+m-1}$, 
	where we use the following notations: 
	\begin{equation}
	  \bm{c}^{i, m; k}\,=\, \left\{
\begin{array}{ll}
(k,c_{1+m},\dots, c_{n+m-1}) & \text{if } i=1,\vspace{1mm}\\
(c_1,\dots, c_{i-1},k,c_{i+m},\dots, c_{n+m-1}) & \text{if } 2\leq i\leq n-1,\vspace{1mm}\\
(c_1,\dots, c_{n-1},k) & \text{if } i=n,
\end{array}
\right.\label{eq1.101}
	\end{equation}
and	
\begin{equation}
\bm{c}_{i,m}= (c_i,\dots, c_{i+m-1}).\label{eq1.102}
\end{equation}
Note that the right hand side of the equation ($\ref{eq1.10}$) is a finite sum since $N(g,\bm{c}_{i,m})<\infty$. 
\end{defi}

\begin{prop}\label{prop2.9}
	The sequence $\W(\C) =\{\W(\C)(n)\}_{n\geq 1}$ with the maps $\circ_i$ defined in Definition\,$\ref{def1.10}$ has the structure of an operad, whose unit $\unit$ is given by 
	\[
	  \unit(d;c)= \left\{
\begin{array}{ll}
1 & \text{if } d=c\\
0 & \text{otherwise}
\end{array}
\right. \hspace{3mm}\text{for any } (d;c)\in \C\times \C.
	\]
\end{prop}

\begin{proof}
	It suffices to show that the associativity and unitality axioms in Definition\,$\ref{def2.6}$ hold. 
	
	\noindent
	{\it Associativity.} \ Let $1\leq j\leq n$, $m, l\geq 1$, $f\in \W(\C)(n)$, $g\in \W(\C)(m) $, and $h\in \W(\C)(l)$. Then, for any $(d;\bm{c})\in\C\times \C^{n+m+l-2}$ we have
	\begin{align}
	 \bigl((f\circ_j g)\circ_i h\bigr)(d;\bm{c})&\,=\,\sum_{k\in\C} \,(f\circ_j g)(d;\bm{c}^{i,l;k})\cdot h(k;\bm{c}_{i,l}) \nonumber\\
	 &\,=\, \sum_{k, k^{\prime}\in\C} f\bigl(d; {(\bm{c}^{i, l; k})}^{j, m; k^{\prime}}\bigr)\cdot g\bigl(k^{\prime}; {(\bm{c}^{i,  l; k})}_{j,m}\bigr)\cdot h(k;\bm{c}_{i, l}).\label{eq2.12}
	\end{align}
	Note that  
	\[
	  (\bm{c}^{i, l; k})^{j, m; k^{\prime}} \,=\, \left\{
\begin{array}{ll}
(\bm{c}^{j+l-1, m; k^{\prime}})^{i, l; k} & \text{if } 1\leq i<j,\\
\bm{c}^{j,m+l-1;k^{\prime}} & \text{if } j\leq i<m+j,\\
(\bm{c}^{j, m; k^{\prime}})^{i-m+1, l; k} & \text{if } j+m\leq i\leq n+m-1,
\end{array}
\right. 
	\]
	\[
	  (\bm{c}^{i, l; k})_{j,m} \,=\, \left\{
\begin{array}{ll}
\bm{c}_{j+l-1,m} & \text{if } 1\leq i<j,\\
(\bm{c}_{j,m+l-1})^{i-j+1, l; k} & \text{if } j\leq i<m+j,\\
\bm{c}_{j,m} & \text{if } j+m\leq i\leq n+m-1,
\end{array}
\right. 
	\]
	and
	\[
	  \bm{c}_{i, l} \,=\, \left\{
\begin{array}{ll}
(\bm{c}^{j+l-1, m; k^{\prime}})_{i, l} & \text{if } 1\leq i<j,\\
(\bm{c}_{j, m+l-1})_{i-j+1, l} & \text{if } j\leq i<m+j,\\
(\bm{c}^{j, m; k^{\prime}})_{i-m+1,l} & \text{if } j+m\leq i\leq n+m-1,
\end{array}
\right. 
	\]
	for any $k, k^{\prime}\in\C$. 
	
	\noindent
	(I)\ The case $1\leq i<j$: Then the right hand side of the equation (\ref{eq2.12}) is rewritten as
	\begin{align*}
		&\sum_{k,k^{\prime}\in\C} f\bigl(d; (\bm{c}^{j+l-1,m;k^{\prime}})^{i,l;k}\bigr)\cdot g\bigl(k^{\prime}; \bm{c}_{j+l-1,m}\bigr)\cdot h\bigl(k; (\bm{c}^{j+l-1,m;k^{\prime}})_{i,l}\bigr)\\&\hspace{30mm}\,=\, \sum_{k^{\prime}\in\C} \,(f\circ_i h)(d; \bm{c}^{j+l-1,m;k^{\prime}})\cdot g\bigl(k^{\prime}; \bm{c}_{j+l-1,m}\bigr)\\
		&\hspace{30mm}\,=\, \bigl((f\circ_i h)\circ_{j+l-1} g\bigr)(d;\bm{c}).
	\end{align*}
	This implies $(f\circ_j g)\circ_i h=(f\circ_i h)\circ_{j+l-1} g$.
	
	\noindent
	(II)\ The case $j\leq i<m+j$: Then the right hand side of the equation (\ref{eq2.12}) is rewritten as
	\begin{align*}
		&\sum_{k,k^{\prime}\in\C} f\bigl(d; \bm{c}^{j,m+l-1;k^{\prime}}\bigr)\cdot g\bigl(k^{\prime}; (\bm{c}_{j,m+l-1})^{i-j+1,l;k}\bigr)\cdot h\bigl(k; (\bm{c}_{j,m+l-1})_{i-j+1,l}\bigr)\\&\hspace{38mm}\,=\, \sum_{k^{\prime}\in\C} f \bigl(d; \bm{c}^{j,m+l-1;k^{\prime}}\bigr)\cdot (g\circ_{i-j+1} h)\bigl(k^{\prime}; \bm{c}_{j,m+l-1}\bigr)\\
		&\hspace{38mm}\,=\, \bigl(f\circ_j (g\circ_{i-j+1} h)\bigr)(d;\bm{c}).
	\end{align*}
	This implies $(f\circ_j g)\circ_i h =f\circ_j (g\circ_{i-j+1} h)$.
	
	\noindent
	(III)\ The case $j+m\leq i\leq n+m-1$: Then the right hand side of the equation (\ref{eq2.12}) is rewritten as
	\begin{align*}
		&\sum_{k,k^{\prime}\in\C} f\bigl(d; (\bm{c}^{j,m;k^{\prime}})^{i-m+1,l;k}\bigr)\cdot g\bigl(k^{\prime}; \bm{c}_{j,m}\bigr)\cdot h\bigl(k; (\bm{c}^{j,m;k^{\prime}})_{i-m+1,l}\bigr)\\&\hspace{28mm}\,=\, \sum_{k^{\prime}\in\C} \,(f\circ_{i-m+1} h)(d; \bm{c}^{j,m;k^{\prime}})\cdot g\bigl(k^{\prime}; \bm{c}_{j,m}\bigr)\\
		&\hspace{28mm}\,=\, \bigl((f\circ_{i-m+1} h)\circ_j g\bigr)(d;\bm{c}).
	\end{align*}
	This implies $(f\circ_j g)\circ_i h=(f\circ_{i-m+1} h)\circ_j g$. Therefore, the associativity axiom holds.
	
	\noindent
	{\it Unitality.} \ Let $1\leq i\leq n$, $m\geq 1$, $f\in \W(\C)(n)$, and $g\in \W(\C)(m) $. Then, for any $(d;\bm{c})\in\C\times\C^n$ and $(b;\bm{a})\in\C\times\C^m$ we have
	\[
	  (f\circ_i\unit)(d;\bm{c})\,=\, \sum_{k\in\C} \,f(d;\bm{c}^{i,1;k})\cdot \unit(k;c_i)
	  \,=\, f(d;\bm{c}),
	\] 
	\[
		(\unit\circ_1 g)(b;\bm{a})\,=\, \sum_{k\in\C} \,\unit(b;k)\cdot g(k;\bm{a})
		\,=\, g(b;\bm{a}).
	\]
	This implies $f\circ_i\unit =f$ and $\unit\circ_1 g =g$. Therefore, the unitality axiom holds. 
	\qed
\end{proof}

We will introduce a morphism of operads below.

\begin{defi}\label{def2.12}
Let $\OO=\{\OO(n)\}_{n\geq 1}$ and $\PP=\{\PP(n)\}_{n\geq 1}$ be (non-symmetric) operads (in the category of sets). {\it A morphism from $\OO$ to $\PP$} is a sequence $\f=\{\f_n: \OO(n)\to \PP(n)\}_{n\geq 1}$ of maps which commute with the operadic compositions and preserve the units, that is, satisfy the following conditions:

\noindent
(i) \ $\f_{n+m-1} (X\circ_i Y)\,=\,\f_n (X)\,\circ_i\,\f_m (Y)$ \ for each $1\leq i\leq n, m\geq 1, X\in\OO(n)$, and $Y\in\OO(m)$,

\noindent
(ii) \ $\f_1 (\unit)\,=\,\unit$.

\noindent
We write $\f: \OO\to\PP$ to indicate that $\f$ is a morphism from $\OO$ to $\PP$. 
\end{defi}

Here are the examples.

\begin{exam}\label{ex2.13}
	For sets $\C$ and $\C^{\prime}$, we consider operads $\W(\C)$ and $\W(\C^{\prime})$ as in Definition\,$\ref{def1.10}$. Let $\Phi:\C\to\C^{\prime}$ be a map. We define the sequence ${\Phi}^* =\{{({\Phi}^*)}_n:\W(\C^{\prime})(n)\to \W(\C)(n)\}_{n\geq 1}$ of maps by 
	\[
	  {({\Phi}^*)}_n(f)\,=\,\Bigl(\,(d;c_1,\dots, c_n)\longmapsto f\bigl(\,\Phi(d);\,\Phi(c_1), \dots, \Phi(c_n)\,\bigr)\,\Bigr).
	\]
	Then ${\Phi}^*$ is a morphism from $\W(\C^{\prime})$ to $\W(\C)$. 
\end{exam}

\begin{exam}\label{ex2.14}
	We define the sequence $\p=\{\p_n: \Tt(n)\to\Tc(n)\}_{n\geq 1}$ of maps by 
	\[
	  \p_n [T]\,=\,n. 
	\]
	Then $\p$ is a morphism from $\Tt$ to $\Tc$. (The map $\p_n$ means contracting all internal edges of trivalent rooted ribbon $n$-trees.) 
\end{exam}

The next lemma gives a criterion for uniqueness of morphisms of operads. 
\begin{lemm}\label{lemm2.14}
	Let $\OO, \PP$ be operads and let $\f, \g :\OO\to\PP$ be two morphisms of operads. We have $\f=\g$ if the following conditions hold.
	
	\noindent
	$(\mathrm{i})$ $\f_1\,=\,\g_1$,
	
	\noindent
	$(\mathrm{ii})$ $\f_2\,=\,\g_2$,
	
	\noindent
	$(\mathrm{iii})$ For $n>2$ and $X\in \OO(n)$, there exist $2\leq m<n$, $2\leq l< n$, $Y\in\OO(m)$, and $Z\in\OO(l)$ such that 
	\[
	  n=m+l-1\text{ \ and \ } X=Y\circ_i Z \text{\,\ for\ some\ } 1\leq i\leq m. 
	\]
\end{lemm}

\begin{proof}
	We show that $\f_n=\g_n$ for any $n\geq 1$ by induction. The cases $n=1,2$ are just the conditions (i) and (ii) respectively. The other cases follow from the condition (iii), Definition\,$\ref{def2.12}$, and the induction hypothesis. 
	\qed
\end{proof}
\begin{cor}\label{cor2.17}
	Let $\PP$ be an operad and let $\f, \g :\Tt\to\PP$ be two morphisms of operads. Then we have $\f=\g$ if $\f_2=\g_2$.
\end{cor}
\begin{proof}
	We check that the conditions\,(i) and (iii) in Lemma\,\ref{lemm2.14} hold. The condition\,(i) holds because $\Tt(1)$ is the singleton of the unit $\unit$ and both $\f$ and $\g$ are morphisms of operads. 
	
	On the other hand, since a trivalent rooted ribbon tree with more than $3$ leaves has always an internal edge, then the condition\,(iii) follows from Proposition\,\ref{prop2.6}.
	\qed
\end{proof}

\section{The K\"ahler polarization}

In Subsection\,$3.1$, we see that the quantum Hilbert space via the K\"ahler polarization can be described as an invariant space of an $SO(3)$-representation, which was referred to in \cite{taka}. Based on this description, we construct the morphism $\fk:\Tc\to\W(\Z_{\geq 0})$ in Subsection\,$3.2$. The keys for this construction are the facts that any $SO(3)$-representation is completely reducible and all irreducible $SO(3)$-representations can be classified with odd numbers. 

\subsection{Quantization via the K\"ahler polarization}

Let $n\geq 3$ and let $\bm{r}=(r_0,\dots,r_{n-1})\in \R^n _{>0}$. First, we specify the K\"ahler structure on the polygon space $\M({\bm{r}})$ as follows. For $i=0,\dots,n-1$, 
we consider the sphere $S^2(r_i)$ as the K\"ahler manifold with the K\"ahler form $\omega_{S^2(r_i)}$ normalized by $\int_{S^2(r_i)}\,\omega_{S^2(r_i)} =2 r_i$. 
Since the standard $SO(3)$-action on $S^2(r_i)$ is Hamiltonian and the momentum map is the inclusion $S^2(r_i)\hookrightarrow \R^3$, then the diagonal $SO(3)$-action on $S^2(r_0)\times\cdots \times S^2(r_{n-1})$ is also Hamiltonian and the momentum map $\mu$ is given by $\mu(\bm{u})\,=\,u_0 +\cdots + u_{n-1}$. Now our polygon space $\M(\bm{r})$ is described as the quotient space
   \[
     \M(\bm{r})\,=\,{\mu}^{-1}(0)/SO(3). 
   \]
Note that the condition that $0$ is a regular value of $\mu$ and  $SO(3)$ acts on ${\mu}^{-1}(0)$ freely is characterized as the following condition on the edge-lengths $\bm{r}$:
   \begin{equation}
  \pm r_0\pm\cdots\pm r_{n-1}\neq 0. \label{eq1}
\end{equation}
Thus we always assume the condition\,(\ref{eq1}) on $\bm{r}$ so that the polygon space $\bigl(\M(\bm{r}), \omega_{ \M(\bm{r})}\bigr)$ can be regard as a smooth K\"ahler quotient of $\bigl(S^2(r_0)\times\cdots \times S^2(r_{n-1}),\, \omega_{S^2(r_0)}\oplus\cdots\oplus \omega_{S^2(r_{n-1})}\bigr)$. 

Now we assume $\bm{r}\in \Z^n _{>0}$ in addition to the condition\,(\ref{eq1}). This integral condition enables us to construct a prequantum line bundle $\Li(\bm{r})\to\M(\bm{r})$ as follows. For $i=0,\dots,n-1$, we denote by $L(r_i)$ the $r_i$-th tensor power of the holomorphic tangent bundle of $S^2(r_i)$, which is a prequantum line bundle over $S^2(r_i)$. Then a prequantum line bundle $L(\bm{r})$ over $S^2(r_0)\times\cdots \times S^2(r_{n-1})$ is given by 
\begin{equation*}
  L(\bm{r})\,=\,{\mathrm{pr}_0}^*\,L(r_0)\otimes\cdots\otimes {\mathrm{pr}_{n-1}}^*\,L(r_{n-1}),\label{eq_2}
\end{equation*}
where $\pr_i$ is the projection $S^2(r_0)\times\cdots \times S^2(r_{n-1})\to S^2(r_i)$. 
Since it is $SO(3)$-equivariant, the bundle $L(\bm{r})$ over $S^2(r_0)\times\cdots \times S^2(r_{n-1})$ descents to a bundle
\begin{equation*}
  \Li(\bm{r})=\left(L(\bm{r})\,\bigl|\,_{\mu^{-1}(0)}\right)\bigl/SO(3)\label{eq_3}
\end{equation*}
over $\M({\bm{r}})$. We find that $c_1(\Li(\bm{r}))=[\omega_{ \M(\bm{r})}]$, namely, the bundle $\Li(\bm{r})$ is a prequantum line bundle over $\M(\bm{r})$. Now we have completed the setting of quantization via the K\"ahler polarization. 

Recall that the quantum Hilbert space $\mathscr{H}_{\mathrm{K}\ddot{\mathrm{a}}\mathrm{h}}$ via the K\"ahler polarization is defined to be the space of holomorphic sections of $\Li(\bm{r})$. 
Next, we rewrite this space as an $SO(3)$-invariant space by using the so-called  ``quantization commutes with reduction'' theorem, which was conjectured by Guillemin and Sternberg\,\cite{GS} and has been proved and improved by several people e.g.\,\cite{brav}, \cite{Mei}, \cite{Tian}. In this paper, we follow a result of Braverman\,\cite{brav}.

Let $(M,\,\omega)$ be a K\"ahler manifold with a holomorphic and Hamiltonian action of a compact Lie group $G$ and let $L$ be a $G$-equivariant prequantum line bundle over $(M,\,\omega)$. We assume that $0$ is a regular value of the moment map $\mu$ and $G$ acts on ${\mu}^{-1}(0)$ freely. Then, we have the K\"ahler quotient $M_G={\mu}^{-1}(0)/G$ and the holomorphic line bundle $L_G$ over $M_G$ such that ${\pi}^* L_G\,=\,L|_{{\mu}^{-1}(0)}$, where $\pi:{\mu}^{-1}(0)\to M_G$ is the natural projection. The ``quantization commutes with reduction'' theorem is the following. 
\begin{theo}{\bf(\cite[Theorem\,1.4]{brav})}\label{thm3.0}
Under the assumption as above, we have
	\[H^j(M_G,L_G)\,=\,{H^j(M,L)}^G\text{ \ for any }j\geq 0. \]
\end{theo}

Now we apply Theorem\,$\ref{thm3.0}$ to the case when 
\[
  G=SO(3),\ (M,L)=\bigl(\,S^2(r_0)\times\cdots \times S^2(r_{n-1}),L(\bm{r})\,\bigr) \text{, and } (M_G,L_G)=\bigl(\M(\bm{r}),\Li(\bm{r})\bigr).
\]
\begin{prop}\label{prop3.1}
	For $\bm{r}\in \Z^n _{>0}$ satisfying the condition\,$(\ref{eq1})$, we have
	\[
	  H^j\bigl(\M(\bm{r}),\Oh_{\Li(\bm{r})}\bigr) = \left\{
\begin{array}{ll}
{\hspace{-1mm}\Bigl(H^0\bigl(S^2(r_0),\Oh_{L(r_0)}\bigr)\otimes\cdots\otimes H^0\bigl(S^2(r_{n-1}),\Oh_{L(r_{n-1})}\bigr)\Bigr)}^{SO(3)} & \text{if } j=0,\vspace{2mm}\\
\hspace{38mm}0 & \text{if } j>0.
\end{array}
\right. 
	\]
\end{prop}

\begin{proof}
	By Theorem\,$\ref{thm3.0}$, we have
	\begin{align*}
	  H^j\bigl(\M(\bm{r}),\Oh_{\Li(\bm{r})}\bigr) &
	  = {\Bigl(H^j\left(S^2(r_0)\times\cdots \times S^2(r_{n-1}),\Oh_{L(\bm{r})}\right)\Bigr)}^{SO(3)}\\
	  &= \Bigl(H^j\bigl(S^2(r_0),\Oh_{L(r_0)}\bigr)\otimes\cdots\otimes H^j\bigl(S^2(r_{n-1}),\Oh_{L(r_{n-1})}\bigr)\Bigr)^{SO(3)}
	\end{align*}
	for any $j\geq 0$. By taking the definition of $L(r_i)$ and positivity of $r_i$ for $i=0,\dots,n-1$ into account, we obtain the proposition. 
	\qed
\end{proof}
\begin{rem}
	When $n$ is odd and $\bm{r}=(1,\dots,1)$, the polygon space is a Fano variety \cite[Corollary\,2.3.3]{kly}. In this case, the vanishing of the higher cohomologies is also obtained from the Kodaira-Nakano vanishing theorem without the ``quantization commutes with reduction'' theorem. 
\end{rem}
\subsection{The morphism of operads associated to the K\"ahler polarization}

First, in the case of compact $\Lie$ groups, we see that multiplicities of an irreducible component in tensor representations give a morphism of operads. 

\begin{prop}\label{prop3.4}
	Let $G$ be a compact Lie group and denote by $\widehat{G}$ the set of all equivalent classes of an irreducible finite-dimensional complex representation of $G$. We define the sequence $\mg=\bigl\{{(\mg)}_n : \Tc(n)\to \W(\widehat{G})(n)\bigr\}_{n\geq 1}$ of maps by
	\[
	  {(\mg)}_n (n)\,= \,\Bigl((\,W ; V_1,\dots,V_n)\longmapsto[\,V_1\otimes\cdots\otimes V_n: W\,]\,\Bigr), 
	\]
	where $[\,V_1\otimes\cdots\otimes V_n: W\,]$ is the multiplicity of an irreducible representation $W$ in $V_1\otimes\cdots\otimes V_n$. Then the sequence $\mg$ is a morphism of operads from $\Tc$ to $\W(\widehat{G})$. 
\end{prop}

\begin{proof}
Let $n, m\geq 1$, $1\leq i\leq n$, $W \in\widehat{G}$, and $\bm{V}=(V_1,\dots,V_{n+m-1})\in{\widehat{G}}^{\,n+m-1}$. Then we have the following equalities:
\begin{align*}
	&\bigl({(\mg)}_{n+m-1} (n\circ_i m)\bigr)(W;\bm{V})\\& \hspace{10mm}\,=\, [\,V_1\otimes\cdots\otimes V_{n+m-1}: W\,]\\
	& \hspace{10mm}\,=\,\sum_{U\in \widehat{G}} [\,V_1\otimes\cdots V_{i-1}\otimes U\otimes V_{i+m}\otimes V_{n+m-1}: W\,]\cdot [\,V_i\otimes\cdots\otimes V_{i+m-1}: U\,]\\
	& \hspace{10mm}\,=\, \sum_{U\in \widehat{G}} \bigl({(\mg)}_n (n)\bigr)(W; \bm{V}^{i,m;U})\cdot \bigl({(\mg)}_m (m)\bigr)(U;\bm{V}_{i,m})\\
	& \hspace{10mm}\,=\, \bigl(\,{(\mg)}_n (n)\, \circ_i \,{(\mg)}_m (m)\,\bigr) (W;\bm{V}),
\end{align*}
where the second equality follows from the irreducible decomposition 
\[
  V_i\otimes\cdots\otimes V_{i+m-1}\,=\, \bigoplus_{U\in \widehat{G}} \,{V_{\mathrm{triv}}}^{\oplus [\,V_i\otimes\cdots\otimes V_{i+m-1}: U\,]}\otimes U.
\]
(Here $V_{\mathrm{triv}}$ is the equivalent class of the trivial irreducible representations.)  At the third equality, we use the notations (\ref{eq1.101}) and (\ref{eq1.102}) for $\bm{V}=(V_1,\dots,V_{n+m-1})$ and $U$. 

On the other hand, it is clear that ${(\mg)}_1(\unit)=\unit$ by definition. (Recall that the unit of $\W(\widehat{G})$ is given in Proposition\,\ref{prop2.9}.) Now the proposition is proved. 
	\qed
\end{proof}

Now, we focus on $SO(3)$-representations. We define the map $\mathrm{Dim}: \widehat{SO(3)}\to \Z_{>0}$ by assigning each equivalent class to its dimension. It is well-known that this map is bijective onto the set of odd numbers.   
\begin{defi}\label{def3.5}
   We define a sequence $\fk=\{{(\fk)}_n : \Tc(n)\to \W(\Z_{\geq 0})(n)\}_{n\geq 1}$ of maps by
	\[
	  {(\fk)}_{n}(n)\,=\,\Bigl((d ; \bm{c})\longmapsto [\,R(c_1)\otimes\cdots\otimes R(c_n) : R(d)\,]\Bigr),
	\]
   where $R$ is the map $\Z_{\geq 0}\to \widehat{SO(3)}$ given by $R(m)=\mathrm{Dim}^{-1}(2m+1)$. 
\end{defi}

\begin{prop}\label{prop3.6}
The sequence $\fk$ is a morphism of operads from $\Tc$ to $\W(\Z_{\geq 0})$. 
\end{prop}

\begin{proof}
The sequence $\fk$ is given by the composition of morphisms of operads
\[
     \fk :\,\Tc\xrightarrow{\ms\ } \W(\widehat{SO(3)})\xrightarrow{\ \ R^*\ } \W(\Z_{\geq 0}), 
   \]
   where $\ms$ is the morphism given in Proposition\,\ref{prop3.4} for $G=SO(3)$ and $R^*$ is the morphism induced by the map $R$ (see Example\,\ref{ex2.13}). 
   This proves the proposition.
   \qed
\end{proof}

\begin{lemm}\label{prop3.7}	
	If $n=2$, we have for $(d;c_1, c_2)\in\Z_{\geq 0} \times \Z^2 _{\geq 0}$
	\[
	  \bigl(\,{(\fk)}_2 (2)\,\bigr)(d;c_1, c_2)\,=\, \left\{
\begin{array}{ll}
1 & \text{if \,} |c_1 - c_2|\leq d\leq c_1 + c_2,\\
0 & \text{otherwise}.
\end{array}
\right. 
	\]
\end{lemm}

\begin{proof}
	By definition, we have
	\[
	  \bigl(\,{(\fk)}_2 (2)\,\bigr)(d;c_1, c_2)\,=\,[\,R(c_1)\otimes R(c_2) : R(d)\,].
	\]
	This multiplicity is computed from the Clebsch-Gordan rule for $SO(3)$ (see e.g. \cite{YS}), which proves the assertion.
	\qed
\end{proof}

Finally, we see that the morphism $\fk:\Tc\to\W(\Z_{\geq 0})$ controls the dimension of the space of holomorphic sections.

\begin{prop}\label{prop3.5}
	Suppose $n\geq 3$. Then we have the following for any $n$-tuple $\bm{r}=(r_0,r_1,\dots, r_{n-1})$ of positive integers satisfying the condition\,$(\ref{eq1})$:
	\[
	  \bigl(\,{(\fk)}_{n-1} (n-1)\,\bigr)(r_0;r_1,\dots, r_{n-1})\,=\,\dim H^0\bigl(\M (\bm{r}),\,\Oh_{\Li (\bm{r})}\bigr). 
	\]
\end{prop}
\begin{proof}
	The Borel-Weil theorem tells us that each $H^0\bigl(S^2(r_i),\Oh_{L(r_i)}\bigr)$ in Proposition\,$\ref{prop3.1}$ is an irreducible $SO(3)$-representation of dimension $2r_i +1$. Therefore, we have
	 \begin{align*}
	  \dim H^0\bigl(\M(\bm{r}),\Oh_{\Li(\bm{r})}\bigr)&\,=\dim\,\bigl(\,R(r_0)\otimes R(r_1)\otimes\cdots\otimes R(r_{n-1})\,\bigr)^{SO(3)}\\
	  &\,=\,\dim\,\bigl(\,R(r_0)^*\otimes R(r_1)\otimes\cdots\otimes R(r_{n-1})\,\bigr)^{SO(3)}\\
      &\,=\,[\,R(r_1)\otimes\cdots\otimes R(r_{n-1}) : R(r_0)\,].
	\end{align*}
	This proves the assertion.
	\qed
\end{proof}

\section{The real polarization}

In Subsection\,$4.1$, we define the bending system associated to any triangulation of polygons by using its dual graph, a trivalent rooted ribbon tree. In addition, we rewrite the number of the associated lattice points to fit our operadic formulation. With this description, we construct the morphism $\fr:\Tt\to\W(\Z_{\geq 0})$ to the real polarization in Subsection\,$4.2$. We also comment on Kamiyama's recurrence relation in our framework. 

\subsection{The bending system}

Let $n\geq 3$ and $\bm{r}=(r_0,\dots,r_{n-1})\in\R^n _{> 0}$. 
As in Figure\,\ref{fig6}, a decomposition of a trivalent rooted ribbon $(n-1)$-tree induces that of an $n$-gon. Then, the length of the new side-edges of two polygons defines a function on the polygon space $\M(\bm{r})$, called the {\it bending Hamiltonian}. 
The precise definition is the following. 
\begin{defi}[Kapovich and Millson\,\cite{kapo}]\label{defi4} 
Let $T$ be a trivalent rooted ribbon $(n-1)$-tree. As in the proof of Proposition\,\ref{prop2.6}, an edge $e\in E(T)$ determines the grafting decomposition of $T$ into rooted ribbon trees $S_e$ and $S_e ^{\prime}$. This also induces the decomposition
 \[\{0,\dots, n-1\}\,=\,I_e\sqcup I_e ^{\prime},\] 
where $I_e$ (resp. $I_e ^{\prime}$) is the set of numbers $i=0,\dots, n-1$ such that the $i$-th external vertex $(v_T)_i$ is contained in $S_e$ (resp. $S_e ^{\prime}$). Then we define the function $b_e : \M(\bm{r})\to \R$ by
\[
  b_e [\bm{u}]\,=\,\Biggl\|\sum_{i\,\in\, I_e} u_i\Biggr\|\,=\,\Biggl\|\sum_{i\,\in\, I_e ^{\prime}} u_i\Biggr\|,
\] 
which is called the {\it bending Hamiltonian}.
\end{defi}

\begin{figure}[!h]
\begin{minipage}[b]{0.49\columnwidth}
\centering
\begin{tikzpicture}[scale=0.7, transform shape, line width=0.7pt]
\draw  (1.2,3)--(1,5)--(3,6.5)--(4.5,6)--(5.3,3)--cycle;
\draw[red] (2,6.5)--(1.73,4.83);
\draw[red] (0.5,4)--(1.73,4.83);
\draw[red] (2.9,5.16)--(1.73,4.83);
\draw[red] (2.9,5.16)--(4,7);
\draw[red] (2.9,5.16)--(3.66,4);
\draw[red] (5.49,5)--(3.66,4);
\draw[red] (3.66,2.3)--(3.66,4);
 \draw (3.2,3)node[below]{$u_2$};
 \draw (1.1,3.8)node[left]{$u_1$};
 \draw (1.8,5.6)node[above left]{$u_0$};
 \draw (3.8,6.2)node[above right]{$u_4$};
 \draw (5,4.5)node[right]{$u_3$};
 \draw[red] (2,6.5)node[above]{$0$};
 \draw[red] (0.5,4)node[left]{$1$};
 \draw[red] (4,7)node[above right]{$4$};
 \draw[red] (5.49,5)node[right]{$3$};
 \draw[red] (3.66,2.3)node[below]{$2$};
\end{tikzpicture}
\end{minipage}
\begin{minipage}[b]{0.49\columnwidth}
\centering
\begin{tikzpicture}[scale=0.7, transform shape, line width=0.7pt]
\draw  (1.2,3)--(1,5)--(3,6.5)--(4.5,6)--cycle;
\draw[red] (2,6.5)--(1.73,4.83);
\draw[red] (0.5,4)--(1.73,4.83);
\draw[red] (2.9,5.16)--(1.73,4.83);
\draw[red] (2.9,5.16)--(4,7);
\draw[red] (2.9,5.16)--(3.4,4.6);
 \draw (1.1,3.8)node[left]{$u_1$};
 \draw (1.8,5.6)node[above left]{$u_0$};
 \draw (3.8,6.2)node[above right]{$u_4$};
 \draw (5.3,5.5)--(2,2.5)--(6.1,2.5)--cycle;
\draw (5.8,4)node[right]{$u_3$};
\draw (4,2.5)node[below]{$u_2$};
\draw[red] (3.63,4.3)--(4.43,3.5);
\draw[red] (4.43,3.5)--(4.43,1.8);
\draw[red] (4.43,3.5)--(6.2,4.5);
\draw[red] (2,6.5)node[above]{$0$};
 \draw[red] (0.5,4)node[left]{$1$};
 \draw[red] (4,7)node[above right]{$4$};
 \draw[red] (6.2,4.5)node[right]{$3$};
 \draw[red] (4.43,1.8)node[below]{$2$};
\end{tikzpicture}
\end{minipage}
\caption{decomposing a polygon.}\label{fig6}
\end{figure}
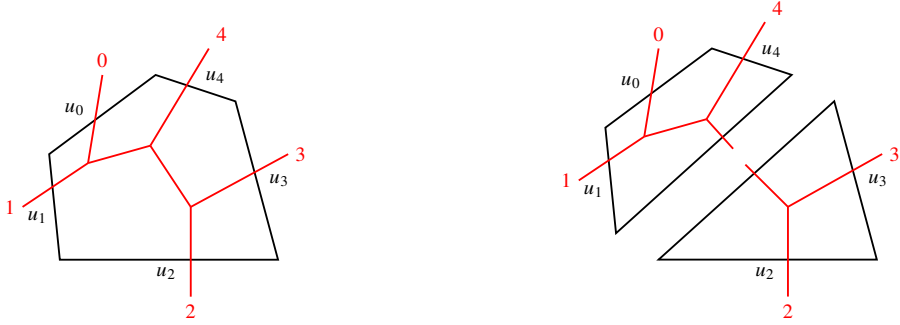
Here are the elemental properties of the bending Hamiltonians. 
\begin{lemm}\label{lem4.2}
	We have the followings on the bending Hamiltonian $b_e$'s.
	
	\noindent
	$(1)$ $|b_{e^{\prime}}-b_{e^{\prime\prime}}|\leq b_e \leq b_{e^{\prime}}+b_{e^{\prime\prime}}$ if $e$, $e^{\prime}$ and $e^{\prime\prime}$ are adjacent to a common vertex.
		
	\noindent
	$(2)$ $b_e$ is the constant function with value $r_i$ if $e$ is adjacent to the $i$-th external vertex.

\end{lemm}
\begin{proof}
(1) Let $[\bm{u}] \in\M(\bm{r})$. Since $I_e\sqcup I^{\prime} _{e^{\prime}}\sqcup I^{\prime} _{e^{\prime\prime}}=\{0,\dots,n-1\}$, we have 
	\[
	  \sum_{i\,\in\, I_e} u_i + \sum_{i\,\in\, I^{\prime} _{e^{\prime}}} u_i + \sum_{i\,\in\, I^{\prime} _{e^{\prime\prime}}} u_i\,=\, u_0+\cdots + u_{n-1}\,=\,0
	\]
	and hence 
	\[\bigl|b_{e^{\prime}}[\bm{u}]-b_{e^{\prime\prime}}[\bm{u}]\bigr|\leq b_e [\bm{u}]\leq b_{e^{\prime}}[\bm{u}]+b_{e^{\prime\prime}}[\bm{u}],\] which proves the assertion. 
	
	\noindent
	(2) Let $[\bm{u}] \in\M(\bm{r})$. If $e$ is adjacent to the $i$-th external vertex, we have $I_e=\{i\}$ or $I^{\prime} _e=\{i\}$ 
	and hence
	\[
	  b_e [\bm{u}]
	  \,=\, \| u_i\|
	  \,=\, r_i,
	\]
	which proves the assertion. 
	\qed
\end{proof}

We are ready to define the {\it bending system associated to a trivalent rooted ribbon tree}. Note that the number of internal edges of a trivalent rooted ribbon $(n-1)$-tree is equal to $n-3$.

\begin{defi}[Kapovich and Millson\,\cite{kapo}]\label{defi4.1}
	Suppose $n\geq 4$ and fix a numbering $\lambda:\{1,\dots,n-3\}\to\Ei(T)$. Then the {\it bending system on $\M(\bm{r})$ associated to $T$} is the collection of the $n-3$ bending Hamiltonians 
	\[
  \pi^{\bm{r}} _{T}=(b_{\lambda(1)},\dots,b_{\lambda(n-3)}):\M(\bm{r})\longrightarrow \R^{n-3}. 
\]
\end{defi}

\begin{rem}\label{rem4.4}
	When $T$ is the $(n-1)$-caterpillar (see Example\,\ref{ex2.6}) and $\lambda$ is given in order of closeness to the root, then the bending system $\pi^{\bm{r}} _{T}$ coincides with the original bending system\,(\ref{eq00}). 
	\end{rem}

The next theorem is a fundamental result on the bending system. 
\begin{theo}[Kapovich and Millson\,\cite{kapo}]\label{thm4.2}
	Suppose that $\bm{r}\in\R^n _{>0}$ satisfies the condition\,$(\ref{eq1})$. 
	Then the bending system $\pi^{\bm{r}} _T:\M(\bm{r})\to\R^{n-3}$ is a completely integrable system on an open dense subset $\M^{\prime} (\bm{r})$ of $\M(\bm{r})$ where it is smooth. Moreover, the Hamiltonian flows generated by the bending Hamiltonians induce a $(n-3)$-dimensional torus action on $\M^{\prime} (\bm{r})$ and then, the moment map is given by the restriction of $\pi^{\bm{r}} _T$ to $\M^{\prime} (\bm{r})$.
\end{theo}

In the sense of Theorem\,$\ref{thm4.2}$, the polygon space with the bending system can be considered as an ``almost'' toric manifold. 

Recall that we consider the dimension of the quantum Hilbert space $\mathscr{H}_{\mathrm{re}}$ via the real polarization as the number of lattice points in the closure of the moment polytope: $\#\,\mathrm{Im} (\pi^{\bm{r}} _T) \cap \Z^{n-3}$.  
In the rest of this subsection, we will rewrite the number $\#\,\mathrm{Im} (\pi^{\bm{r}} _T) \cap \Z^{n-3}$ as $\#\D(T,\bm{r})$, the number of integral edge-labelings of $T$ given in Definition\,\ref{defi4.6} below. Here, we consider any $n$-tuple $\bm{r}=(r_0,\dots,r_{n-1})$ of {\it non-negative} integers for the following reason: when we prove that $\fr$ given in Definition\,\ref{def4.10} becomes a morphism of operads in Proposition\,\ref{prop4.16}, we need to consider grafting of trees with non-negative integral labelings (see Lemma\,\ref{lem4.5}). In terms of triangulations, this corresponds to consider ``gluing'' of triangulated polygons with non-negative edge-lengths. Thanks to including the case of zero-labeling, we can include triangulated polygons after contracting some diagonals in the argument as well. 

Here is the definition of integral edge-labelings of trees mentioned above.
\begin{defi}\label{defi4.6}
	Let $n\geq 2$, $\bm{r}=(r_0,\dots,r_{n-1})\in\Z^n _{\geq 0}$, and let $T$ be a trivalent rooted ribbon $(n-1)$-tree. 
	An {\it admissible integral labeling of $T$ relative to $\bm{r}$} is a function $\varphi: E(T)\to\Z_{\geq 0}$ with the following property:
	
	\noindent
	$(1)$ $\bigl|\varphi(e^{\prime})-\varphi(e^{\prime\prime})\bigr| \leq \varphi(e) \leq \varphi(e^{\prime})+ \varphi(e^{\prime\prime})$ if $e$, $e^{\prime}$ and $e^{\prime\prime}$ are adjacent to a common vertex. 
	  
\noindent
$(2)$ $\varphi(e)=r_i$ if $e$ is adjacent to the $i$-th external vertex.
	
	\noindent
	We denote by $\D(T,\bm{r})$ the set of admissible integral labelings. 
\end{defi}

\begin{rem}
 Our integer labelings of a trivalent rooted ribbon tree are some modifications of those of a pants decomposition of a compact Riemann surface with boundary given by Jeffrey and Weitsman\,\cite[Definition 4.8, 4.9]{JW2}. The next proposition is an analog of Theorem\,4.10(b) in \cite{JW2}.
\end{rem}

\begin{prop}\label{prop4.10}
	Suppose that $n\geq 4$ and $\bm{r}\in\Z_{> 0} ^n$ satisfies $\M(\bm{r})\neq \emptyset$. Then we have $\#\D(T,\bm{r})=\#\,\mathrm{Im} (\pi^{\bm{r}} _T) \cap \Z^{n-3}$. 
\end{prop}

We will prove this proposition in the next section.

\subsection{The morphism of operads associated to the real polarization}

Before we prove Proposition\,\ref{prop4.10}, we first define the morphism $\fr$ in Definition\,\ref{def4.10} by using the set of admissible integral labelings introduced in Definition\,\ref{defi4.6}, and prove that it is indeed a morphism of operads. After that, we prove Proposition\,\ref{prop4.10} so that the morphism $\fr$ describes the number of lattice points. 

First of all, we note the following proposition which will be proved later. 
\begin{prop}\label{prop4.6}
	Let $T$ be a trivalent rooted ribbon $n$-tree and let $(d;\bm{c})\in \Z_{\geq 0}\times \Z^n _{\geq 0}=\Z^{n+1} _{\geq 0}$. Then the followings hold.
	
	\noindent
	$(1)$ $\D(T,\, d;\bm{c})=\emptyset$ \ if $d>|\bm{c}|=c_1+\cdots+c_n$.
	
	\noindent
	$(2)$ $\#\D(T,\,d;\bm{c})<\infty.$
\end{prop}

Proposition\,$\ref{prop4.6}$ allows us to make the next definition.
\begin{defi}\label{def4.10}
	We define a sequence $\fr=\{{(\fr)}_n : \Tt(n)\to \W(\Z_{\geq 0})(n)\}_{n\geq 1}$ of maps by
	\[
	  {(\fr)}_{n}[T]\,=\,\Bigl((d ; \bm{c})\longmapsto \#\D(T,\,d;\bm{c})\Bigr).
	\]
\end{defi}

Then we show the following proposition.
\begin{prop}\label{prop4.16}
	The sequence $\fr$ is a morphism of operads from $\Tt$ to $\W(\Z_{\geq 0})$.
\end{prop}

To prove Propositions\,\ref{prop4.6} and \ref{prop4.16}, we prepare a couple of lemmas.

\begin{lemm}\label{lem4.4}
	Let $T$ be a trivalent rooted ribbon $n$-tree. 
	
	\noindent
	$(1)$ If $n=1$, we have for $(d;c)\in \Z_{\geq 0}\times \Z_{\geq 0}$
	\[
	  \#\D(T,\,d;c)\,=\, \left\{
\begin{array}{ll}
1 & \text{if \,} d=c,\\
0 & \text{otherwise}.
\end{array}
\right. 
	\]
	
	\noindent
	$(2)$ If $n=2$, we have for $(d;c_1, c_2)\in \Z_{\geq 0}\times \Z^2 _{\geq 0}$
	\[
	  \#\D(T,\,d;c_1, c_2)\,=\, \left\{
\begin{array}{ll}
1 & \text{if \,} |c_1 - c_2|\leq d\leq c_1 + c_2,\\
0 & \text{otherwise}.
\end{array}
\right. 
	\]
\end{lemm}
\begin{proof}
	If $n=1$ (resp. $n=2$), then $T$ is nothing but the exceptional tree (resp. the $2$-corolla) in Example\,\ref{ex1.2}. Therefore, the assertions follow from Definition\,$\ref{defi4.6}$.
	\qed
\end{proof}
\begin{lemm}\label{lem4.5} 
Let $T$ and $S$ be trivalent rooted ribbon trees, let $n$ (resp. $m$) be the number of leaves of $T$ (resp. $S$), and let $1\leq i\leq n$. Then, the set $\D\bigl(T\circ_i S,\,d;\bm{c}\bigr)$ is in bijective correspondence with
\[\bigsqcup_{k\in \Z_{\geq 0}} \D\bigl(T,\,d;\bm{c}^{i,m;k}\bigr)\times \D\bigl(S,\,k;\bm{c}_{i,m}\bigr)
	\]
	for each $(d;\bm{c})\in \Z_{\geq 0}\times \Z^{n+m-1} _{\geq 0}$. Recall that $\bm{c}^{i,m;k}$ (resp. $\bm{c}_{i,m}$) is the $n$-tuple (resp. the $m$-tuple) of non-negative integers from the notations\,(\ref{eq1.101}) and (\ref{eq1.102}). 
\end{lemm}
\begin{proof}
	The bijective correspondence is given by
	\[
	  \varphi\longmapsto (\varphi|_{E(T)},\,\varphi|_{E(S)}),
	\]
	where we identify $T$ and $S$ with subcomplexes of $T\circ_i S$ as in Definition\,$\ref{def1.5}$. Indeed, by setting a number $k$ as the value of $\varphi$ at the internal edge $(-(r_T)_i)\sqcup (-r_S)\in \Ei(T\circ_i S)$, we have
	\[\varphi|_{E(T)}\in \D \bigl(T,\,d;\bm{c}^{i,m;k}\bigr)\text{ \ and \ }\varphi|_{E(S)} \in \D \bigl(S,\,k;\bm{c}_{i,m}\bigr).\]
	\qed
\end{proof}

Now, using the two lemmas above, we prove Propositions\,\ref{prop4.6} and \ref{prop4.16}.
\begin{shomeio}
	We prove both (1-2) by induction on $n$ respectively. The cases $n=1,2$ of (1-2) follow from Lemma\,$\ref{lem4.4}$. We assume that $n\geq 3$. Then, since $T$ always has an internal edge, there exist $m,m^{\prime}\geq 2$, a ribbon $m$-tree $S$, a ribbon $m^{\prime}$-tree $S^{\prime}$, and $1\leq i\leq m$ such that $T$ is isomorphic to $S\circ_i S^{\prime}$ by Proposition\,$\ref{prop2.6}$.
	
	\noindent
	(1) We assume that there exists an admissible integral labeling $\varphi$ of $T=S\circ_i S^{\prime}$ relative to $(d;\bm{c})$. 
	As in the proof of Lemma\,\ref{lem4.5}, we find that
	\[\varphi|_{E(S)}\in \D \bigl(S,\,d;\bm{c}^{i,m;k}\bigr)\text{ \ and \ }\varphi|_{E(S^{\prime})} \in \D \bigl(S^{\prime},\,k;\bm{c}_{i,m}\bigr)\] 
	for some $k\in \Z_{\geq 0}$ and hence, we have
	\[
	 d\leq |\bm{c}^{i,m,k}|\text{ \ and \ }k\leq |\bm{c}_{i,m}|
	\]
	by the induction hypothesis. Therefore we obtain 
	\[
	  d\leq |\bm{c}^{i,m,k}| =c_1+\dots+c_{i-1}+k+c_{i+m}+\dots+c_{n+m-1}\leq |\bm{c}|
	\] as desired.
	
	\noindent
	(2) By Lemma\,$\ref{lem4.5}$ and Proposition\,$\ref{prop4.6}$(1), we have
	\[
	 \#\D\bigl(S\circ_i S^{\prime},\,d;\bm{c}\bigr)=\sum_{k=0}^{|\bm{c}_{i,m}|}\#\D \bigl(S,\,d;\bm{c}^{i,m,k}\bigr)\cdot\#\D \bigl(S^{\prime},\,k;\bm{c}_{i,m}\bigr).
	\]
	Hence, we obtain $\#\D\bigl(S\circ_i S^{\prime},\,d;\bm{c}\bigr)<\infty$ by the induction hypothesis.
	\qed
\end{shomeio}

\begin{shomeiu}
	Let $T$ and $S$ be rooted ribbon trees, let $n$ (resp. $m$) be the number of leaves of $T$ (resp. $S$), let $1\leq i\leq n$, and let $(d;\bm{c})\in \Z_{\geq 0}\times \Z^{n+m-1} _{\geq 0}$. Then we have
	\begin{align*}
	  \bigl(\,{(\fr)}_{n+m-1}\,[T]\circ_i[S]\,\bigr)(d;\bm{c}) &\,=\, \#\D\bigl(T\circ_i S,\,d;\bm{c}\bigr)\\
	  &\,=\, \sum _{k\in \Z_{\geq 0}} \#\D \bigl(T,\,d;\bm{c}^{i,m;k}\bigr)\cdot\#\D \bigl(S,\,k;\bm{c}_{i,m}\bigr)\\
	  &\,=\, \sum _{k\in \Z_{\geq 0}} \bigl({(\fr)}_n [T]\,\bigr) (d;\bm{c}^{i,m;k})\cdot \bigl({(\fr)}_m [S]\,\bigr) (k;\bm{c}_{i,m})\\
	  &\,=\,\bigl(\,{(\fr)}_n [T]\, \circ_i \,{(\fr)}_m [S]\, \bigr)(d;\bm{c}),
	\end{align*}
	where the second equality is due to Lemma\,$\ref{lem4.5}$. 
	
	On the other hand, we obtain ${(\fr)}_1 \unit=\unit$ from Lemma\,$\ref{lem4.4}.(1)$. (Recall that the unit of $\Tt$ is given by the tree with one leaf and the unit of $\W(\Z_{\geq 0})$ is given in Proposition\,\ref{prop2.9}.) Now the proposition is proved.
	\qed
\end{shomeiu}

We show that the morphism $\fr:\Tt\to\W(\Z_{\geq 0})$ controls the number of the lattice points associated to the bending system.
\begin{prop}\label{prop4.19}
	Suppose $n\geq 4$. We have the following for any trivalent rooted ribbon $(n-1)$-tree $T$ and any $n$-tuple $\bm{r}=(r_0,r_1,\dots, r_{n-1})$ of positive integers satisfying $\M(\bm{r})\neq\emptyset$: 
	\[
	  \bigl(\,{(\fr)}_{n-1} [T]\,\bigr) (r_0;r_1,\dots, r_{n-1})\,=\,\#\,\mathrm{Im} (\pi^{\bm{r}} _T) \cap \Z^{n-3}.
	\]
\end{prop}
\begin{proof}
	The proposition is proved as follows:
	\[
	  \bigl(\,{(\fr)}_{n-1} [T]\,\bigr) (r_0;r_1,\dots, r_{n-1})\,=\, \#\D(T,\bm{r})\,=\,\#\,\mathrm{Im} (\pi^{\bm{r}} _T) \cap \Z^{n-3},
	\]
	where the second equality is due to Proposition\,\ref{prop4.10}. 
	\qed
\end{proof}

Now we start to prove Proposition\,\ref{prop4.10}. Hereafter, we will generalize the underlying set of the polygon space to take polygons with contracted side-edges into account. 
Let $n\geq 2$ and $(d;\bm{c})\in \Z_{\geq 0}\times \Z^n _{\geq 0}=\Z^{n+1} _{\geq 0}$. Then, we consider the following set:
\[
  \M(d;\bm{c})=\bigl\{\bm{u}=(u_0,u_1,\dots,u_n)\in S^2(d)\times S^2(c_1)\times\cdots \times S^2(c_n)\,\bigl|\, u_0 +\cdots + u_n = 0\bigr\}/SO(3),
\]
where $S^2(0)$ is the point $\{0\}\subset\R^3$ with the trivial $SO(3)$-action and the quotient is taken by the diagonal action. Note that $\M(d;\bm{c})\neq\emptyset$ if and only if
\begin{equation}
 d\leq c_1+\dots+c_n\text{ \ and \ } c_i\leq d + \sum_{j\neq i} c_j\text{ for each }i=1,\dots,n.\label{eq4.17}
\end{equation}
First, we note the following lemma. Let $T$ be a trivalent rooted ribbon $n$-tree. 

\begin{lemm}\label{lem4.15}
	We have $\M(d;\bm{c})\neq\emptyset$ if $\D(T,\,d;\bm{c})\neq\emptyset$.
\end{lemm}
\begin{proof}
	We assume that $\D(T,\,d;\bm{c})\neq\emptyset$. Note that
	\begin{align*}
		\D(T,\,d;\bm{c})&\,=\, \D(T_1,\,c_1;c_2,\dots,c_n,d)\\
		&\,=\, \cdots\,=\,\D(T_n,\,c_n;d,c_1,\dots,c_{n-1}),
	\end{align*}
	where $T_i$ is the rooted ribbon $n$-tree where the underlying ribbon tree structure is the same as $T$ but the $i$-th leaf $(r_T)_i$ of $T$ is regarded as the root. Thus, applying Proposition\,\ref{prop4.6}(1) for each side of the equations above, we obtain the inequalities\,(\ref{eq4.17}). This proves the lemma.
	\qed
\end{proof}

To prove Propositions\,\ref{prop4.10}, we prepare some definitions and lemmas below.

\begin{defi}\label{defi4.15}
	Suppose $\M(d;\bm{c})\neq\emptyset$. 
	
	\noindent
	(1) We define the functions $b_e:\M(d;\bm{c})\to\R$ for any edge $e\in E(T)$ in the same way as Definition\,\ref{defi4}. 
	
	\noindent
	(2) We define the following set:
	\[
	 B(T,\,d;\bm{c})\,=\,\{\varphi:E(T)\to \Z_{\geq 0}\,\bigl|\,\varphi=b_{\bullet}[\bm{u}]\text{ for some polygon } [\bm{u}]\in\M(d;\bm{c})\,\},
	\]
	where $b_{\bullet} [\bm{u}]$ means the function $(e\mapsto b_e [\bm{u}]\,)$. 
	
	\noindent
	(3) If $n\geq 3$, we also define the map $\pi^{d;\bm{c}} _T:\M(d;\bm{c})\to\R^{n-2}$ in the same way as Definition\,\ref{defi4.1}. (Note that the number $\#\Ei(T)$ is equal to $n-2$ since $T$ has $n$ leaves.)

\end{defi}
\begin{lemm}\label{lem4.22}
Suppose $\M(d;\bm{c})\neq\emptyset$. The functions $b_e$'s given in Definition\,\ref{defi4.15}(1) also have the same properties (1-2) in Lemma\,\ref{lem4.2}, that is, we have the followings:

\noindent
$(1)$ $|b_{e^{\prime}}-b_{e^{\prime\prime}}|\leq b_e \leq b_{e^{\prime}}+b_{e^{\prime\prime}}$ if $e$, $e^{\prime}$ and $e^{\prime\prime}$ are adjacent to a common vertex.
		
	\noindent
	$(2)$ $b_e$ is the constant function with value $d$ (resp. $c_i$) if $e$ is adjacent to the root vertex (resp. the $i$-th leaf vertex).

\end{lemm}

\begin{proof}
The argument in the proof of Lemma\,\ref{lem4.2} is valid even if $u_i=0$ for some $i$. 
\qed
\end{proof}
\begin{lemm}\label{lemm4.4}
	If $\M(d;\bm{c})\neq\emptyset$ and $n\geq 3$, then we have $\#\,B(T,\,d;\bm{c})=\#\,\mathrm{Im} (\pi^{d;\bm{c}} _T) \cap \Z^{n-2}$.
\end{lemm}
\begin{proof}
	Let $\lambda:\{1,\dots,n-2\}\to\Ei(T)$ be a fixed numbering. We find that the bijective correspondence is given by the following map: 
	\[
	B(T,\,d;\bm{c})\longrightarrow \mathrm{Im} (\pi^{d;\bm{c}} _T) \cap \Z^{n-2}\ ;\,\ \varphi\longmapsto \bigl((\varphi\circ\lambda)(1),\dots,(\varphi\circ\lambda)(n-2)\bigr).
	\]
	Indeed, the values $\varphi(e)$ for any $e\in E(T)\setminus\Ei(T)$ are determined by $d$ or $\bm{c}=(c_1,\dots, c_n)$ from the property (2) in Lemma\,\ref{lem4.22}, which shows that the map above is bijective.   
	\qed
\end{proof}

\begin{lemm}\label{lem4.12}
	Suppose $\M(d;\bm{c})\neq\emptyset$. Then we have $\D(T,\,d;\bm{c})= B(T,\,d;\bm{c})$. 
\end{lemm}
\begin{proof}
	By the properties (1-2) in Lemma\,\ref{lem4.22}, we have $\D(T,\,d;\bm{c})\supset B(T,\,d;\bm{c})$. Hereafter, we prove $\D(T,\,d;\bm{c})\subset B(T,\,d;\bm{c})$ by induction on $n\geq 2$. 
	
	First, we prove the case $n=2$. Let $\varphi$ be an admissible integral labeling of $T$ relative to $(d;c_1,c_2)$. Then it follows from Lemma\,\ref{lem4.4}(2) that $|c_1-c_2| \leq d \leq c_1 + c_2$, which guarantees existence of the triangle $[\bm{u}]\in\M(d;c_1,c_2)$. It is clear that $\varphi=b_{\bullet}[\bm{u}]$.
	
	From now on, we assume that the proposition holds for any $2\leq n^{\prime}<n$. Then there exist $m,m^{\prime}\geq 2$, a ribbon $m$-tree $S$, a ribbon $m^{\prime}$-tree $S^{\prime}$, and $1\leq i\leq m$ such that $T$ is isomorphic to $S\circ_i S^{\prime}$ by Proposition\,$\ref{prop2.6}$. Let $\varphi$ be an admissible integral labeling of $T=S\circ_i S^{\prime}$ relative to $(d;\bm{c})$. As in the proof of Lemma\,\ref{lem4.5}, we find that
	\[\varphi|_{E(S)}\in \D \bigl(S,\,d;\bm{c}^{i,m;k}\bigr)\text{ \ and \ }\varphi|_{E(S^{\prime})} \in \D \bigl(S^{\prime},\,k;\bm{c}_{i,m}\bigr)\] 
	for some $k\in \Z_{\geq 0}$ (may be $k=0$) and hence, we have
	\[\M(d;\bm{c}^{i,m^{\prime};k})\neq\emptyset\text{ \ and \ }\M(k;\bm{c}_{i,m^{\prime}})\neq\emptyset\] by Lemma\,\ref{lem4.15}. 
	Applying the induction hypothesis for $S$ and $(d;\bm{c}^{i,m^{\prime};k})$, and for $S^{\prime}$ and $(k;\bm{c}_{i,m^{\prime}})$ respectively, we have two polygons \[[\bm{v}=(v_0,\dots,v_m)]\in \M(d;\bm{c}^{i,m^{\prime};k})\text{ \ and \ }[\bm{w}=(w_0,\dots,w_{m^{\prime}})]\in\M(k;\bm{c}_{i,m^{\prime}})\] satisfying 
	\[
	 \varphi|_{E(S)} \,=\,b_{\bullet}[\bm{v}]\text{ \ and \ } \varphi|_{E(S^{\prime})} \,=\,b_{\bullet}[\bm{w}].
	\]
	Recall the notations\,(\ref{eq1.101}) and (\ref{eq1.102}) for $\bm{c}=(c_1,\dots, c_n)$ and $k$ again. Since $\|v_i\|=k=\|w_0\|$, we can take an element $g\in SO(3)$ satisfying $v_i=-\, g\cdot w_0$ (in the case $k=0$, we can take any $g\in SO(3)$) and 
	define the following $(n+1)$-gon: 
	\[
	  \bm{u}\,=\,(v_0,\dots,v_{i-1},\, g\cdot w_1,\dots,\, g\cdot w_{m^{\prime}},v_{i+1},\dots,v_m).
    \]
	It is easy to see that $\bm{u}$ is indeed a $(n+1)$-gon with edge-lengths $(d;\bm{c})$ and $\varphi=b_{\bullet}[\bm{u}]$ (see Figure\,\ref{fig7}). Now we have shown $\D \bigl(S\circ_i S^{\prime},\,d;\bm{c}\bigr)\subset B\bigl(S\circ_i S^{\prime},\,d;\bm{c}\bigr)$. Thus we obtain $\D(T,\,d;\bm{c})\subset B(T,\,d;\bm{c})$. 
	\qed
\end{proof}

\begin{figure}[!h]
\centering
\begin{tikzpicture}[scale=0.6, transform shape, line width=0.7pt]
\draw  (4.7,6.2)--(3,6.5)--(1,5)--(1.2,3)--(2.7,2.1);
\draw (1.1,3.8)node[left]{$v_1$};
 \draw (1.8,5.6)node[above left]{$v_0$};
 \draw (3.7,6.4)node[above right]{$v_m$};
 \draw (5.7,2.9)node[below right]{$v_i=-g\cdot w_0$};
 \draw[loosely dotted] (2.7,2.1)--(5,2.2);
 \draw[loosely dotted] (4.7,6.2)--(6.3,3.7);
\draw (5,2.2)--(6.3,3.7);
\draw (6.3,3.7)--(8,3.7)--(9.3,2.2);
\draw (5,2.2)--(5.6,0.3)--(7,-0.2);
\draw (7.4,3.8)node[above]{$g\cdot w_{m^{\prime}}$};
\draw (5.2,1.2)node[left]{$g\cdot w_1$};
 \draw[loosely dotted] (9.1,2)--(7.1,-0.1);
\end{tikzpicture}
\caption{the polygon $\bm{u}$.}\label{fig7}
\end{figure}
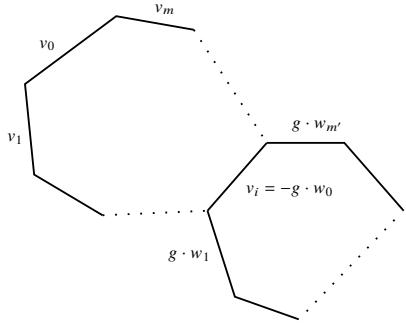

Now we are ready to prove Proposition\,\ref{prop4.10}.
\begin{shomeiy}
	Replacing $n$ in Lemmas\,\ref{lemm4.4} and \ref{lem4.12} by $n-1$, and $(d;\bm{c})$ by $(r_0;r_1,\dots, r_{n-1})=\bm{r}$, then we obtain 
	\[
	  \#\D(T,\bm{r})\,=\,\#B(T,\bm{r})\,=\,\#\,\mathrm{Im} (\pi^{\bm{r}} _T) \cap \Z^{n-3}.
	\]
	\qed
\end{shomeiy}

At the end of this subsection, we derive Kamiyama's recurrence relation mentioned in Section\,$1$ from our morphism $\fr:\Tt\to\W(\Z_{\geq 0})$. 

\begin{rem}\label{rem4.20}
	In the proof of Theorem\,\ref{thm0}, Kamiyama considered the bending system $\pi^{(i,1,\dots,1)}:\M(i,1,\dots,1)\to\R^{n-3}$ for $n\geq 4$ and $1\leq i\leq n-1$, and derived the recurrence relation for the number $\beta_{n,i}$ of the lattice points in its image. Here is the recurrence relation:	
	\begin{equation}
	  \beta_{n,i}\,=\, \left\{
\begin{array}{ll}
\beta_{n-1,1} & \text{if \,} i=0,\\
\beta_{n-1,i-1}+ \beta_{n-1,i}+ \beta_{n-1,i+1} & \text{if \,} 1\leq i\leq n-1,\\
0 & \text{if \,} n\leq i,
\end{array}
\right. \label{eq4.20}
	\end{equation}
	for any $n\geq 4$ and $i\geq 0$. 
	
	In our operadic formulation, Kamiyama's recurrence relation\,(\ref{eq4.20}) arises from the fact that caterpillars have the canonical grafting decomposition: 
	\begin{equation}
	C_{n-1} = C_2\circ_2 C_{n-2} \text{ \ for each } n\geq 4, \label{eq4.200}
	\end{equation}
	where $C_n$ is the isomorphism class of the $n$-caterpillars in Example\,\ref{ex2.6}. 
	
	First, by Remark\,\ref{rem4.4} and Proposition\,\ref{prop4.19}, we have  
	\[
	  \beta_{n,i}\,=\,\bigl(\,{(\fr)}_{n-1}\,C_{n-1}\,\bigr)(i;1,\dots,1)
	\]
	for any $n\geq 4$ and $i\geq 0$. Therefore, we have $\beta_{n,i}=0$ if $n\leq i$ by Proposition\,\ref{prop4.6}(1). 
	
	On the other hand, we have 
	\begin{align}
	  &\bigl(\,{(\fr)}_{n-1}\,C_{n-1}\,\bigr)(i;1,\dots,1)\notag\\
	  & \hspace{20mm} =\bigl(\,{(\fr)}_{n-1}\,C_2\circ_2 C_{n-2}\ \bigr)(i;1,\dots,1)
	  \notag\\& \hspace{20mm}= \bigl(\,\bigl({(\fr)}_2\,C_2\bigr)\circ_2 \bigl({(\fr)}_{n-2}\,C_{n-2}\bigr)\,\bigr)(i;1,\dots,1)\vspace{1mm}\notag\\
	  & \hspace{20mm} =\sum_{k\in \Z_{\geq 0}} \bigl(\,{(\fr)}_2\,C_2 \,\bigr)(i;1,k)\cdot \bigl(\,{(\fr)}_{n-2}\,C_{n-2} \,\bigr)(k;1,\dots,1)\label{eq4.21}
	\end{align}
	and 
	\begin{equation}
	  \bigl(\,{(\fr)}_2\,C_2 \,\bigr)(i;1,k)\,=\, \left\{
\begin{array}{ll}
1 & \text{if \,} |i-1|\leq k\leq i+1,\\
0 & \text{otherwise \,}
\end{array}
\right. \label{eq4.22}
	\end{equation} 
	by Proposition\,\ref{lem4.4}(2). Therefore, we have the relation\,(\ref{eq4.20}) for the other cases where $0\leq i\leq n-1$. 
	
	Thus, we have reproduced Kamiyama's recurrence relation\,(\ref{eq4.20}) using our operadic formulation.
\end{rem}

\section{The proof of the main theorem}

Now we complete the proofs of Theorem\,\ref{thm1.1} and Corollary\,\ref{cor1.2}.\,\,Recall that $\p$ is the morphism of operads given in Example\,\ref{ex2.14}. 
\begin{shomei}
	For the K\"aler polarization on the polygon spaces, we have constructed the morphism of operads $\fk:\Tc\to\W(\Z_{\geq 0})$ in Definition\,\ref{def3.5} and Proposition\,\ref{prop3.6}. 
	
	On the other hand, for the real polarization on the polygon spaces, we have also constructed the morphism of operads $\fr:\Tt\to\W(\Z_{\geq 0})$ in Definition\,\ref{def4.10} and Proposition\,\ref{prop4.16}. 
	
	Furthermore, we have ${(\p^* \fk)}_2={(\fr)}_2$ by Lemmas\,\ref{prop3.7} and \ref{lem4.4}(2) and hence we obtain $\p^* \fk=\fr$ by Corollary\,\ref{cor2.17}. 
Now the proof of Theorem\,\ref{thm1.1} is completed. 
	\qed
\end{shomei}
\begin{shomeip}
	Let $T$ and $\bm{r}=(r_0,\dots,r_{n-1})$ as in Corollary\,\ref{cor1.2}. By Propositions\,\ref{prop3.5} and \ref{prop4.19}, the both sides have been described as
	\[\dim H^0\bigl(\M(\bm{r}),\Oh_{\Li(\bm{r})}\bigr)=\bigl(\,{(\fk)}_{n-1} (n-1)\,\bigr) (r_0;r_1,\dots, r_{n-1}),\]  
	\[\#\,\mathrm{Im} (\pi^{\bm{r}} _T) \cap \Z^{n-3}=\bigl(\,{(\fr)}_{n-1} [T]\,\bigr) (r_0;r_1,\dots, r_{n-1}).\] 
	
	On the other hand, it follows from Theorem\,\ref{thm1.1} that
		\[
	{(\fk)}_{n-1} (n-1)\, =\,{(\p^* \fk)}_{n-1} [T]\,=\,{(\fr)}_{n-1} [T].
	\]
	Therefore, we obtain the corollary. 
	\qed
\end{shomeip}

Graduate School of Mathematics, Nagoya University, Nagoya, Japan

\textit{E-mail address:} \texttt{m17024d@math.nagoya-u.ac.jp}

\end{document}